\documentclass[11pt,a4paper,twosided]{article}
\usepackage[utf8]{inputenc}
\usepackage[T1]{fontenc}
\usepackage{amsfonts}
\usepackage{amssymb}
  \usepackage[pdftex]{graphicx}
  \usepackage[pdftex]{color}
  \usepackage[bookmarksopen=false,pdftex=true,breaklinks=true,backref=page,%
     pagebackref=true,plainpages=false,%
     hyperindex=true,pdfstartview=FitH,colorlinks=true,pdfpagelabels=true,%
     colorlinks=true,linkcolor=blue,citecolor=red
     ]%
  {hyperref}

\input{AMSsym.def}
\def\mathbb{\Bbb}

\DeclareSymbolFont{lasy}{U}{lasy}{m}{n}
\SetSymbolFont{lasy}{bold}{U}{lasy}{b}{n}
\let\Box\undefined
\DeclareMathSymbol\Box{\mathord}{lasy}{"32}

\newtheorem{theorem}{Theorem}[section]
\newtheorem{lemma}[theorem]{Lemma}
\newtheorem{proposition}[theorem]{Proposition}
\newtheorem{remark}[theorem]{Remark}

\newtheorem{corollary}[theorem]{Corollary}
\newtheorem{notation}[theorem]{Notation} \newtheorem{definition}{Definition}

\newenvironment{proof}[1]{
\trivlist \item[\hskip \labelsep{\bf #1}]}{\hfill\mbox{$\Box$} \endtrivlist}
\newcommand {\junk}[1]{}
\def\={\hbox{\ \bf =}\ }

\def\.@{\char'76}

\def \Z{{\mathbb Z}}

\def \Q{{\mathbb Q}}

\def \gx{{\bf{x}}}
\def \gu{{\bf{u}}}

\def \gt{{\bf{t}}}
\def \gz{{\bf{z}}}

\def \gy{{\bf{y}}}

\def \ni {\noindent}

\def \ch {{\cal H}}
\def \ci {{\cal I}}
\def \cc {{\cal C}}
\def \cl {{\cal L}}
\def \cd {{\cal D}}
\def \cm {{\cal M}}
\def \cp {{\cal P}}

\def \cv {{\cal V}}
\def \dg {{\cal DG}}

\def \nz {\neq 0}
\def \Vr {{\rm Vr}}
\def \U {{\rm U}}
\def \Rn {{\rm Rn}}
\def \Pos {> 0}
\def \Nng {\ge 0}

\def \Rzero {R_{=0}}
\def \Rnz {R_{\nz}}
\def \Rvr {R_{\Vr}}
\def \Ru {R_{\U}}
\def \Rrn {R_{\Rn}}
\def \Rpos {R_{\Pos}}
\def \Rnng {R_{\Nng}}

\def \zg {{\Z[G]}}
\def \Izero {{\ci}_{=0}}
\def \Mnz {{\cm}_{\nz}}
\def \Mpos {{\cm}_{\Pos}}
\def \Cnng {{\cc}_{\Nng}}
\def \Irn {{\ci}_{\Rn}}
\def \Vvr {{\cv}_{\Vr}}
\def \Mu {{\cm}_{\U}}
\def \abg {{\rm Ab}(G)}
\def \Hzero {{\ch}_{=0}}
\def \Pnng {{\cp}_{\Nng}}

\begin{document}

\title{Dynamical method in algebra: Effective Nullstellens\"atze} \author{
Michel Coste
\thanks{IRMAR (UMR CNRS 6625), Universit\'{e} de Rennes 1, Campus de Beaulieu
35042 Rennes cedex FRANCE, coste@maths.univ-rennes1.fr, supported in part
by European Community contract CHRX-CT94-0506}, \\Henri Lombardi
\thanks{
Laboratoire de Math\'ematiques,  UMR CNRS 6623, UFR des Sciences et Techniques,
Universit\'e de Franche-Comt\'e, 25 030 BESANCON cedex, FRANCE,
henri.lombardi@univ-fcomte.fr,
supported in part by the project
ESPRIT-BRA 6846POSSO},
\\Marie-Fran\c{c}oise Roy\thanks{IRMAR (UMR CNRS  6625), Universit\'{e} de
Rennes 1,
Campus de Beaulieu 35042 Rennes cedex FRANCE, mfroy@maths.univ-rennes1.fr
supported in
part by the project
ESPRIT-BRA 6846POSSO and by European
Community contract CHRX-CT94-0506}
}
\date{revised, May 2025}
\maketitle

This is the arXiv version, arXiv:1701.05794, written in May 2025. 

First one minor change w.r.t.\ the paper in Annals of Pure and Applied Logic 111, (2001) 203-256: a piece after the end of Theorem 5.7 has been put inside the theorem.

Second,  Proposition \ref{3.17} was not correct. In fact Axiom $S(6)_{of}$ of quasi-ordered rings is not provable in the theory or ordered fields as defined in the original paper.  Corollary  \ref{3.15} was correct but the proof had a gap for the same reason. In the present version, we have fixed this issue by adding a new axiom $S(3)_{of}$ in the axioms of ordered fields (so axioms of quasi-ordered rings change their names and begin with $S(4)_{of}$ instead of $S(3)_{of}$). In Lemma \ref{Lemme 3.4} we have added a case corresponding to $S(3)_{of}$.

\begin{abstract}

We give a general method for producing various effective Null and
Positivstellens\"atze, and getting new
Positivstellens\"atze in algebraically closed valued fields and
ordered groups.
These various effective
Nullstellens\"atze produce algebraic identities certifying that some
geometric conditions cannot be simultaneously satisfied. We produce also
constructive versions of abstract classical results of algebra based on Zorn's
lemma in several cases where such constructive version did not exist. For
example, the fact that a real field can be totally ordered, or the fact that a
field can be embedded in an algebraically closed field.
Our results are based on the concepts we develop of dynamical proofs and
simultaneous collapse.
\end{abstract}



\newpage

\tableofcontents

\newpage

\section*{Introduction}
\addcontentsline{toc}{section}{Introduction}

Our aim is to interpret constructively non-constructive classical algebraic
proofs.
The idea is that there is a constructive content hidden in the proof of
theorems
like
``a ring with a non-trivial ideal has a prime ideal", ``a field can be
embedded
in an algebraically closed fields"
even if their proof is based on Zorn's lemma. The constructive content is the
following ``rings with non-trivial ideal and fields collapse simultaneously",
``fields and algebraically closed fields collapse simultaneously" : if
facts can
be shown to be contradictory inside the theory of algebraically closed fields,
using dynamical proofs, they are contradictory as well inside
the theory of non-trivial rings,
and the second contradiction can be explicitly constructed from the first one.
Dynamical proofs are particularly simple: you want to prove a fact in a field
and you do not know whether a given element is null or invertible. You just
open
branches corresponding to the two possible cases and prove this fact in all
subcases. It turns out that many classical
algebraic proofs have this very simple structure.

A similar example is the following
``a real field can be embedded in an ordered field", to be replaced
by ``real fields and ordered fields collapse simultaneously".
Constructively, we
are not able to say that there exists a model of the stronger theory
extending a
model of the weaker one, but only that working with the stronger theory
does not
create more contradiction than working with the weaker one.

In the particular cases that we consider here, simultaneous collapse takes
very
explicit forms, and produces algebraic certificates, which are precisely
algebraic identities
given by various effective
Nullstellens\"atze and
Positivstellens\"atze. We consider Hilbert's Nullstellensatz, and Stengle's
Positivstellensatz, as well as new Positivstellens\"atze for algebraically
closed valued fields and ordered groups, with the same method. Here is the
statement for valued fields: \par\vspace{6pt}
\ni{\bf Theorem
(Positivstellensatz for algebraically closed valued fields)} {\it Let
$(K,A)$ be
a
valued field and let $U_A$ the invertible elements of $A$, $I_A$ the maximal
ideal of $A$.
Suppose that $(K',A')$ is an algebraically closed valued field extension of
$K$
(so that $A=A' \cap K$). Denote by $U_{A'}$ the invertible elements of $A'$,
$I_{A'}$ the maximal ideal of $A'$.

\ni Consider five finite families $(\Rzero,\Rnz,\Rvr,\Rrn,\Ru)$ of elements of
the polynomial ring
$K[x_1,x_2,\ldots,x_m]=K[x]$.
Let $\Izero$ be the ideal of $K[x]$ generated by $\Rzero$, $\Mnz$ the monoid of
$K[x]$ generated by $\Rnz$, $\Vvr$ the subring of $K[x]$ generated by $\Rvr
\cup
\Rrn \cup \Ru \cup A$, $\Irn$ the ideal of $\Vvr$ generated by $\Rrn\cup I_A$,
$\Mu$ the monoid generated by $\Ru\cup U_A$.

\ni Define ${\cal S} \subset {K'}^m$ as the set of points satisfying
the following conditions:
$n(x) = 0$ for $n \in \Rzero$,
$t(x) \not= 0 $ for $t \in \Rnz$,
$c(x) \in A' $ for $c \in \Rvr$,
$v(x) \in U_{A'}$ for $v \in \Ru$, $k(x) \in I_{A'}$ for $k \in \Rrn$.

The set ${\cal S}$ is empty if and only if there is an equality $$m (u+j)
+i \=
0$$
with $m\in \Mnz$, $u \in \Mu$, $j \in \Irn$ and $i \in \Izero$. }
\par\vspace{6pt}
The statement has a trivial part: if there is an equality $$m (u+j) + i = 0$$
with $m \in \Mnz$, $u \in \Mu$, $j \in \Irn$ and $i \in \Izero$ it is clear
that
${\cal S}$ is empty.
The converse implication, from the geometric fact ``${\cal S}$ is empty" to
the
existence of an algebraic identity $$m(u+j) + i = 0$$
with $m \in \Mnz$, $u \in \Mu$, $j \in \Irn$ and $i \in \Izero$, is far from
trivial.
The fact that moreover this algebraic identity can be explicitly constructed
from a proof that ${\cal S}$ is empty is the main point in the present paper.

\smallskip The previous theorem is closely related to results of Prestel and
Ripoli \cite{PrRi}. We
discuss this point in section 4.4.

\smallskip This paper is a first step in a general program of
constructivization
of
classical abstract algebra using dynamical methods (see also
\cite{Lom94,Lom96,Lom98b,lom99,lom99a,lc00,lq99}).

Our theory has many connections with the following papers
(\cite{DDD85,DD89,Duv89,DG93,DR90,DR931,DR932,DS,Gom93}), based on
(\cite{BE72,CL84,CL88,Ehr68}), and with the theory of coherent toposes as
well.

We thank the referee for valuable remarks and bibliographical references.

\section{Dynamical proofs}

Consider the following proof of $x^3-y^3=0 \,\vdash\, x-y=0$ in the theory of
ordered fields.

Suppose that $x^3-y^3=0$. There are two cases to consider

\begin{itemize}
\item $x=0$, then $y^3=0$, and it follows $y=0$, hence $x-y=0$, \item $x^2
> 0$
then $x^3-y^3=(x-y)(x^2+xy+y^2)=(x-y)(3x^2/4+(y+x/2)^2)$ and since $x^2>0$,
$(3x^2/4+(y+x/2)^2)>0$.
Introducing the inverse $z$ of $(3x^2/4+(y+x/2)^2)$ and multiplying
$x^3-y^3$ by
$z$ we see that $x-y=0$. \end{itemize}

This proof is the prototype of a dynamical proof as we shall see soon.

\subsection{Dynamical
theories and dynamical proofs}

We start from a language $\cl$ with variables, constants, symbols of functions
and symbols of relations, including at least the equality. All the theories we
shall consider will allow the substitution of equal terms.
A {\em presentation} in
the language $\cl$ is a couple
$(G;R)$ where $R$ is a set of atomic formulas and $G$ is a set of variables
containing the variables appearing in $R$. The variables in $G$ are called the
{\em generators} and the atomic formulas in $R$ are called the {\em relations}
of the presentation. The sets $G$ and $R$ are allowed to be infinite, but
only a
finite part of them is used in proofs.

A {\em fact} in a presentation $(G;R)$
is any atomic formula of $\cl$ involving only variables in $G$.

A {\em model} of a presentation $(G;R)$ is a set-theoretic interpretation
$A$ of
the language
$\cl$ and a mapping $f$ from $G$ to $A$ such that the relations of $R$ are
valid
inside $A$ after substituting variables $x$ in $G$ by $f(x)$.

We say that {\em the presentation $(G,R)$ contains the presentation $(G',R')$}
when $G'\subset G$ and $R'\subset R$. The {\em union of two presentations}
$(G;R)$ and $(G';R')$ is the presentation $(G \cup G';R \cup R')$ and will be
also denoted by $(G;R) \cup (G';R')$. More generally we use the notation $(G;R)
\cup (G';R')$ in case that
$(G;R)$ is a presentation and relations in $R'$ are relations about terms
constructed on $G\cup G'$.

To a set-theoretic interpretation $A$ of the language $\cl$ one associates the
{\em diagram} of
$A$, $\dg(A)$ which is the presentation where every element $a$ of $A$ is
represented by a variable $X_a$ and the relations are all atomic formulas true
inside $A$.
Remark that $\dg(A)$ does not contain negations of atomic formulas (it is
often called
 the positive diagram of $A$). So e.g. in the theory of rings, the fact
that
two elements $a$ and $b$ of a ring $A$ are distinct does not appear in the
diagram of $A$.

A {\em dynamical theory} $\cd$ has {\em dynamical axioms} i.e.
axioms of the form
$$ H(\gx) \,\vdash\, \exists \gy_1 \ A_1(\gx,\gy_1)\ \vee \ \cdots\ \vee \
\exists \gy_k \ A_k(\gx,\gy_k) $$
where $H(\gx)$ and $A_i(\gx,\gy_i)$ are conjunctions of atomic formulas of
$\cl$, $\gx$ and $\gy_i$ are lists of variables. These theories are also known
in categorical logic under the name of {\em coherent theories}, because of
their
connection with coherent toposes (see Section \ref{topos}).

A special kind of dynamical axiom is an axiom with empty disjunction, denoted
$\bot $, on the right-hand side.
An axiom with $\bot$ in the
right-hand side and a conjunction of variable-free atomic formulas on the left-
hand
side is called a {\em collapse axiom}.

An {\em algebraic theory} ${\cal T}$ has only {\em algebraic}
{\em axioms} i.e. axioms of the form
$$H(\gx) \,\vdash\, K(\gx)$$
where $H(\gx)$ is a conjunction of atomic formulas and $K(\gx)$ is an atomic
formula of $\cl$.

A {\em purely equational theory} is an algebraic theory with only {\em purely
equational axioms} i.e. axioms of the form $$\vdash t=t'$$ where $t$ and $t'$
are terms of the language.

For example the theory of commutative rings is a purely equational theory, the
theories of fields, of algebraically closed fields, of ordered fields, or real
closed fields are dynamical theories.

\medskip
A {\em covering of the presentation $(G;R)$} in the dynamical theory $\cd$
is a
tree constructed in the following way:

\begin{itemize}
\item at each node $n$ of the tree, there is a {\em presentation} $(G_n;R_n)$,
where $G_n$ is the disjoint union of $G$ and a finite set of new
generators, and
$R_n$ is the union of $R$ and a finite set of new relations,
\item at the root $[0]$ of the tree, the presentation is $(G;R_0)$,
where the new relations are consequences of $R$ under algebraic axioms of
${\cal D}$,
\item new nodes are created only in the following way:
if $\gt$ is a list of terms in the variables of $G_n$,
$H(\gt)$
is a conjunction of relations in $R_n$ and $$ H(\gx) \,\vdash\, \exists \gy_1
A_1(\gx,\gy_1)\ \vee\ \cdots\ \vee\ \exists \gy_k \ A_k(\gx,\gy_k) $$
is an axiom of $\cd$, then one can create $k$ new nodes $[n,1],\ldots,[n,k]$
(note
that $k$ may be $0,1$ or $>1$) taking $G_{n,i}=G_n \cup \{\gz_i\}$ (where
variables $\gz_i$ are new in the branch), and $R_{n,i}$ contains $R_{n,i}'=R_n
\cup \{A_i(\gt,\gz_i)\}$ and some consequences of $R_{n,i}'$ under the
algebraic
axioms of $\cd$.
\end{itemize}

A {\em dead branch} of the tree is one ended by an empty disjunction $\bot$. A
{\em leaf} of the tree is a terminal node of a non-dead branch.

 \begin{definition} \label{def-
proof}
A {\em dynamical proof} in $\cd$ of a fact $B(\gt)$ in a presentation
$(G;R)$ is
a covering of $(G;R)$ for the theory $\cd$ where $B(\gt)$ is a valid fact at
every leaf
of the tree, i.e.,
$B(\gt)$ is one of the relations in the presentation at this leaf. We say that
this is a dynamical proof in the theory $\cd$ of $\ R \,\vdash\, B(\gt)$.
\end{definition}

Note that a dynamical proof can be represented by a finite object: it is
sufficient to keep in $(G;R)$ only the generators and relations that are
used in
the proof.
Remark also that dynamical proofs involve only ato\-mic relations of the
language.
Moreover, the
``logical part" of a dynamical proof is nothing but direct applications of the
dynamical
axioms (where variables are replaced by terms). So there are some drastic
restrictions on dynamical proofs when compared to usual proofs. This is the
reason why some algebraic consequences are more easily deduced from dynamical
proofs than from usual ones.

In dynamical proofs, we can use some {\em valid dynamical rules}, i.e.,
deduction rules used in the same way as dynamical axioms, and provable from
the
axioms of the dynamical theory. A valid dynamical rule is of the type
$$ H(\gx) \,\vdash\, \exists \gy_1 A_1(\gx,\gy_1)\ \vee\ \cdots\ \vee\ \exists
\gy_k \ A_k(\gx,\gy_k)\; . $$ It is provable in $\cd$ if there is a
covering of the presentation $(\gx;H)$ in $\cd$ such that every leaf of this
covering contains a valid fact $A_i(\gx,\gt_i)$, for some $i$ and some list of
terms $\gt_i$.

\medskip
Let us construct the tree of our prototype dynamical proof of $x^3-y^3=0
\,\vdash\, x-y=0$ in ordered fields.
We use in particular the following properties of ordered fields:
$$\begin{array}{rlccc}
&\vdash& x^2\ge 0&P(1)\\
x>0,\; y\ge 0&\vdash& x+y>0&P(2)\\
x^2 =0 &\vdash &x=0 &P(3)\\
&\vdash& x=0 \ \vee \ x^2 >0 &P(4)\\
x>0&\vdash & \exists z \ zx-1 =0 &P(5)\\ \end{array}$$

The tree consists of four nodes:
\begin{itemize}
\item
The root of the tree: $[0]$ where the generators are $(x,y)$ and the relations
are $(x^3-y^3=0)$. Under the root, there are two nodes $[l]$ and $[r]$, using
$P(4)$ ($x=0$ or $x^2>0$). \begin{itemize}
\item
$[l]$ where the generators are $(x,y)$ and the relations are $(x^3-
y^3=0,\,x=0,\, -y^3=0,\,y^4=0,\,y^2=0,\,y=0,\, x-y=0)$, (the last relations are
a consequence of the first two,
using $P(1),P(2),P(3)$ and computations in rings). \item
$[r]$ with presentation $((x,y);(x^3-y^3=0, x^2>0,T>0))$ where
$T=3x^2/4+(y+x/2)^2=x^2+xy+y^2$. The fact $T>0$ follows from $P(1)$ and
$P(2)$.
Under this node there is another node $[rd]$, where the inverse of $T$ has
been
added according to $P(5)$.
\begin{itemize}
\item
$[rd]$ with presentation
$((x,y,u);(x^3-y^3=0, x^2>0, T>0, uT=1,x-y=0))$, since $(x-y)=uT(x-y)=u(x^3-
y^3)=0$.
\end{itemize}
\end{itemize}
\end{itemize}

$$
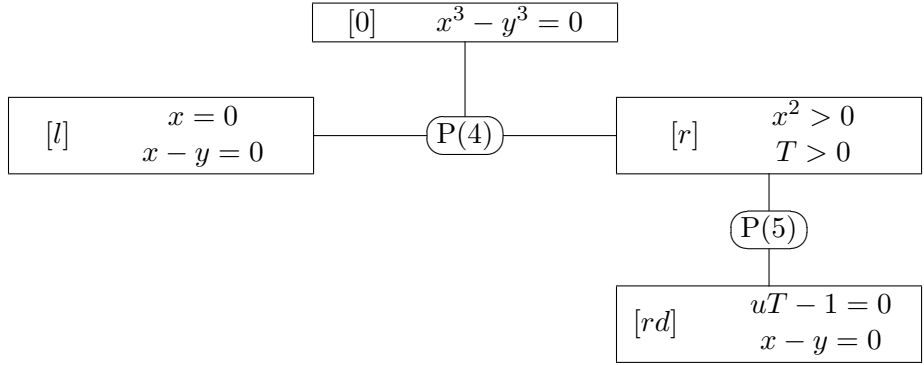
\setlength{\unitlength}{1mm}
\begin{picture}(120,50)
\put(0,27.5){\framebox(40,10){$[l]\qquad\begin{array}{cc} x=0\\
x-y=0
\end{array}$}}
\put(40,32.5){\line(1,0){15}}
\put(60,32.5){\oval(10,5)} \put(60,32.5){\makebox(0,0){P(4)}}
\put(60,35){\line(0,1){10}}
\put(40,45){\framebox(40,5){$[0]\qquad x^3-y^3=0$}}
\put(65,32.5){\line(1,0){15}}
\put(80,27.5){\framebox(40,10){$[r]\qquad\begin{array}{cc} x^2>0\\ T>0
\end{array}$}}
\put(100,27.5){\line(0,-1){5}}
\put(100,20){\oval(10,5)} \put(100,20){\makebox(0,0){P(5)}}
\put(100,17.5){\line(0,-1){5}}
\put(80,2.5){\framebox(40,10){$[rd]\qquad\begin{array}{cc} uT-1=0\\ x-y=0
\end{array}$}}
\end{picture}
$$

Dynamical proofs prove facts and dynamical rules which are obviously valid in
the first order theory $\cd$. Actually they have the same strength as usual
first order logic, for what may be expressed in this fragment.
\begin{theorem} \label{Theoreme 1.3}
Let $\cd$ be a dynamical theory in the language $\cl$, $(G;R)$ a
presentation and $B(\gt)$ a fact of $(G,R)$. There is a construction
associating
to
every proof of $R \,\vdash\, B(\gt)$ in the classical first order theory
$\cd$ a
dynamical proof of $B(\gt)$. \end{theorem}
\begin{proof}{Proof (sketch):}
In a dynamical theory, some elementary predicates are given in the
language, but
it is
not possible to construct predicates using all logical connectives and
quantifiers of first order logic.
In order to get the full strength of usual first order theories, it is
necessary
to allow these constructions of predicates and their use with correct logical
rules.

In this sketch of proof, we describe the introduction of new predicates
corresponding
to disjunction, existential quantifier and classical negation (with the law of
the excluded middle). In each case, we prove that the correct use of a new
predicate does not change provable facts. In classical logic (with the law of
the excluded middle), all the predicates can be introduced with only these
three
constructions. So, if we have a classical proof of a fact, we shall
consider two
distinct dynamical theories, the first
one is the given dynamical theory, the second one is a dynamical theory where
all
predicates used in the classical proof have a name as individual
predicates. The
classical proof is a dynamical proof in the second theory (with convenient
axioms).
Then deleting the new predicates one after the other, beginning by the more
intricate
ones, we see that when dealing with facts of the first dynamical theory,
the two
dynamical theories prove the same facts.

The first two lemmas about disjunction and existential quantification are very
easy.

\begin{lemma}
\label{lem-add-disjunction} Assume that we have a dynamical theory $\cd$ with
some predicates $Q_1,\ldots,Q_k$.
Consider a new dynamical theory $\cd'$, with one
more predicate $Q$, expressing the disjunction of $Q_1,\ldots,Q_k$ and the
following axioms:
$$\begin{array}{rlccl}
Q_1 &\,\vdash \, Q &\qquad &\qquad &Disj_{In,1}(Q_1,\ldots,Q_k,Q)\\ \ldots
&\,\vdash \,
\ldots &\qquad &\qquad &\ldots\\ Q_k &\,\vdash \, Q &\qquad &\qquad
&Disj_{In,k}(Q_1,\ldots,Q_k,Q)\\
Q &\,\vdash \, Q_1\,\vee\,\ldots\,\vee\,Q_k &\qquad &\qquad
&Disj_{El}(Q_1,\ldots,Q_k,Q)\\
\end{array}$$
The dynamical theories $\cd$ and $\cd'$ prove the same facts that do not
involve
the predicate $Q$. 
\end{lemma}
\begin{proof}{Proof:}
We remark that a fact $Q({\gt})$ in a
dynamical proof inside $\cd'$ can
appear only after
an application of an axiom $Disj_{In,j}(Q_1,\ldots,Q_k,Q)$ for some $j$ ($1\le
j\le k$). Consider the first use of the axiom $Disj_{El}(Q_1,\ldots,Q_k,Q)$
in the considered proof tree. It is clear that, if the predicate $Q$ had not
been
introduced, we could get a simpler proof with only the branch corresponding to
$Q_j$
(replacing the $k$ branches appearing after the use of
$Disj_{El}(Q_1,\ldots,Q_k,Q)$).
\end{proof}

\begin{lemma}
\label{lem-add-existential-quantifier} Assume that we have a dynamical theory
$\cd$ with some predicate $R({\gx},y)$.
Consider a new dynamical theory $\cd'$, with one more predicate $S({\gx})$,
expressing
$\exists y\, R({\gx},y)$,
and the following axioms, where $t$ is an arbitrary term:

$$\begin{array}{rlccl}
R({\gx},t) &\,\vdash \, S({\gx}) &\qquad &\qquad &Exis_{In}(R,y,S,t)\\
S({\gx})
&\,\vdash \,\exists y\, R({\gx},y) &\qquad &\qquad &Exis_{El}(R,y,S)\\
\end{array}$$
The dynamical theories $\cd$ and $\cd'$ prove the same facts that do not
involve
the predicate $S$. \end{lemma}
 \begin{proof}{Proof:} We remark
that a fact $S({\gu})$ in a dynamical proof inside $\cd'$ can appear only
after
an application of an axiom
$Exis_{In}(R,y,S,t)$. Consider the first use of the axiom
$Exis_{El}(R,y,S)$ in
the
considered proof tree.
It is clear that, if the predicate $S$ had not been introduced, we could get
another
proof by
replacing $y$ by the term $t$ that allowed its introduction. \end{proof}

\smallskip
The most difficult part of the proof is the following lemma about negation.

 \begin{lemma}
\label{lem-add-negation}
Assume that we have a dynamical theory $\cd$ with some predicate $T({\gx})$.
Consider a new dynamical theory $\cd'$, with one more predicate $F({\gx})$,
expressing the negation of $T({\gx})$ and the following axioms:
$$\begin{array}{rlccl}
&\,\vdash \, T({\gx})\ \vee\ F({\gx}) &\qquad &\qquad &Neg_{In}(T,F)\\
T({\gx}),\ F({\gx}) &\,\vdash \,\bot &\qquad &\qquad &Neg_{El}(T,F)\\
\end{array}$$
The dynamical theories $\cd$ and $\cd'$ prove the same facts that do not
involve
the predicate $F$. \end{lemma}
\begin{proof}{Proof:}
Let us consider a fact $\,\vdash\, A({\gu})$ (where ${\gu}$ is a list of terms)
involving a predicate $A$ distinct from $F$. Let us assume it is proved in the
dynamical theory
$\cd'$.
We have to transform this dynamical proof of $\vdash A({\gu})$ in another one,
with no use of the predicate $F$.
The proof tree of $\vdash A({\gu})$ has dead branches and branches ending with
the fact $A({\gu})$.
The predicate $F$ is used in the proof by creating dead nodes, using the axiom
$Neg_{El}(T,F)$.
Consider one such dead node, and assume w.l.o.g. that the use of this axiom is
the leftmost one
in the proof-tree. (Here we assume that the tree is organized in such a manner
that at each use of the axiom $Neg_{In}(T,F)$ the branch with $T$ is the left-
branch and the branch
with $F$ is the right-branch.)
We call $n$ the dead node, i.e., the node to which $Neg_{El}(T,F)$ is applied.

It suffices to prove that we can transform the proof tree and suppress this
use of $Neg_{El}(T,F)$. Remark first that, if $F$ is present in the tree at
the
left of $n$, it is
useless in this part of the tree and it can be suppressed
(i.e., the occurrences of $Neg_{In}(T,F)$ at the left of $n$ can be
suppressed,
keeping only the right branch).

Since $F({\gt})$ (where ${\gt}$ is a list of terms) is valid at the node
$n$, it
has necessarily been introduced by the use of axiom $Neg_{In}(T,F)$ at some
node
$m$ above $n$. Two branches have thus been opened at the node $m\;$: the left
one with $T({\gt})$, the right one with $F({\gt})$. The subtree ${\cal A}$
under
$T({\gt})$ contains no use of $Neg_{El}(T,F)$
and proves $A({\gu})$ from $T({\gt})$ inside $\cd$. On the other hand,
$T({\gt})$ is a valid fact at the node $n$. So we proceed as follows:
\begin{itemize}
\item Suppress the use of $Neg_{In}(T,F)$ at the node $m$ and keep only the
right branch, suppressing $F(t)$ on the path between $m$ an $n$ and at the
left
of this path.
\item Introduce the use of $Neg_{In}(T,F)$ at the beginning of each branch
opened at the right of the path between $m$ an $n$, gluing the subtree ${\cal
A}$ as the left branch after this use of $Neg_{In}(T,F)$. \item Suppress
the use
of $Neg_{El}(T,F)$ at the node $n$, and glue the subtree ${\cal A}$ under this
node.
\end{itemize}
\end{proof}

These three lemmas complete the sketch of the proof of the theorem. \end{proof}

We give now an example of elimination of negation. We consider the theory of
ordered domains expressed with the only unary predicates
$x=0$ and $x \ge 0$.

The axioms we use are:
$$\begin{array}{rclccl}
x = 0 &\,\vdash\,& xy = 0 &\qquad&\qquad&Alg(1,x,y)\\ x = 0 &\vdash& x \ge 0
	&\qquad&\qquad&Alg(2,x)\\
x \ge 0,\, -x \ge 0 &\vdash& x = 0 &\qquad&\qquad& Alg(3,x)\\
x \ge 0,\, y \ge 0 &\vdash& x+y \ge 0 &\qquad&\qquad&Alg(4,x,y)\\ x \ge 0,\, y
\ge 0 &\vdash& xy \ge 0 &\qquad&\qquad& Alg(5,x,y)\\ & \vdash& x \ge 0 \vee -x
\ge 0 &\qquad&\qquad& Dyn(1,x)\\
xy = 0 &\vdash& x = 0 \vee y = 0 &\qquad&\qquad& Dyn (2,x,y)\\ \end{array}$$

The predicate $x > 0$ is introduced as the predicate opposed to $- x \ge 0$
by the two defining axioms
$$\begin{array}{rclccl}
& \,\vdash\,&
-x \ge 0 \vee x > 0 &\qquad &\qquad&N_{In}(x)\\
-x \ge 0,\, x > 0 &\,\vdash\,& \bot	&\qquad&\qquad&
N_{El}(x)\qquad
\\
\end{array}$$

and it will be used to prove
$$\begin{array}{rclccc}
x+y \ge 0 ,\, xy \ge 0 &\,\vdash& x \ge 0\qquad\qquad
&\qquad&\qquad&\qquad\qquad\qquad
\end{array}$$

We want to transform the following proof (this is surely not a clever one):

$$\setlength{\unitlength}{1mm}
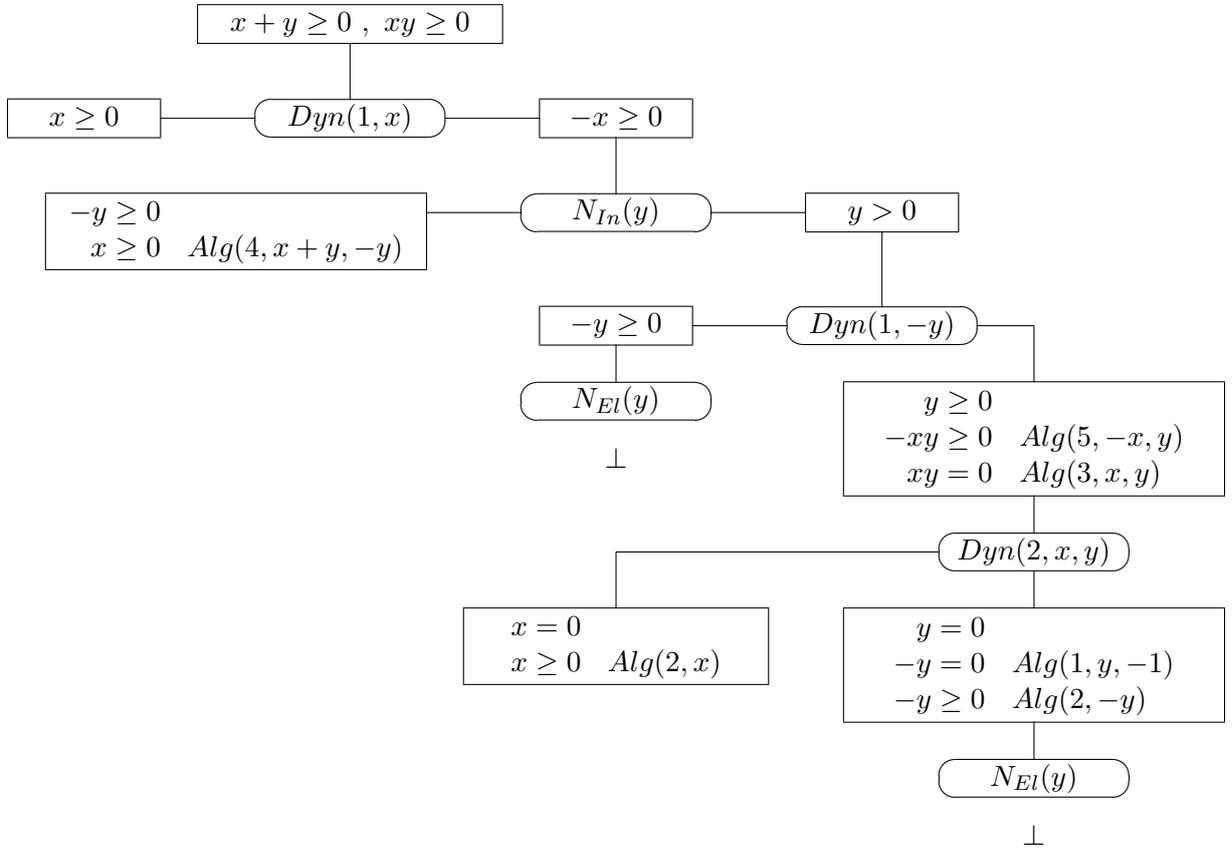
\begin{picture}(160,115)
\put(25,110){\framebox(40,5){$x+y\geq0\ ,\ xy\geq0$}} \put(45,110){\line(0,-
1){7.5}}
\put(45,100){\oval(25,5)} \put(45,100){\makebox(0,0){$Dyn(1,x)$}}
\put(32.5,100){\line(-1,0){12.5}}
\put(0,97.5){\framebox(20,5){$x\geq0$}}
\put(57.5,100){\line(1,0){12.5}}
\put(70,97.5){\framebox(20,5){$-x\geq0$}} \put(80,97.5){\line(0,-1){7.5}}
\put(80,87.5){\oval(25,5)} \put(80,87.5){\makebox(0,0){$N_{In}(y)$}}
\put(67.5,87.5){\line(-1,0){12.5}}
\put(5,80){\framebox(50,10){$\begin{array}{rl} -y\geq0&\\ x\geq0&
Alg(4,x+y,-y)
\end{array}$}}
\put(92.5,87.5){\line(1,0){12.5}}
\put(105,85){\framebox(20,5){$y>0$}}
\put(115,85){\line(0,-1){10}}
\put(115,72.5){\oval(25,5)} \put(115,72.5){\makebox(0,0){$Dyn(1,-y)$}}
\put(102.5,72.5){\line(-1,0){12.5}}
\put(70,70){\framebox(20,5){$-y\geq0$}}
\put(80,70){\line(0,-1){5}}
\put(80,62.5){\oval(25,5)} \put(80,62.5){\makebox(0,0){$N_{El}(y)$}}
\put(80,55){\makebox(0,0){$\bot$}}
\put(127.5,72.5){\line(1,0){7.5}}
\put(135,72.5){\line(0,-1){7.5}}
\put(110,50){\framebox(50,15){$\begin{array}{rl}
y\geq0&\\-xy\geq0&Alg(5,-x,y)\\
xy=0&Alg(3,x,y) \end{array}$}} \put(135,50){\line(0,-1){5}}
\put(135,42.5){\oval(25,5)} \put(135,42.5){\makebox(0,0){$Dyn(2,x,y)$}}
\put(122.5,42.5){\line(-1,0){42.5}}
\put(80,42.5){\line(0,-1){7.5}}
\put(60,25){\framebox(40,10){$\begin{array}{rl} x=0&\\ x\geq0& Alg(2,x)
\end{array}$}}
\put(135,40){\line(0,-1){5}}
\put(110,20){\framebox(50,15){$\begin{array}{rl} y=0&\\ -y=0& Alg(1,y,-1)\\ -
y\geq0&Alg(2,-y)
\end{array}$}}
\put(135,20){\line(0,-1){5}}
\put(135,12.5){\oval(25,5)} \put(135,12.5){\makebox(0,0){$N_{El}(y)$}}
\put(135,5){\makebox(0,0){$\bot$}}
\end{picture}$$
We proceed in two steps. First we suppress the leftmost occurrence of
$N_{El}(y)$.

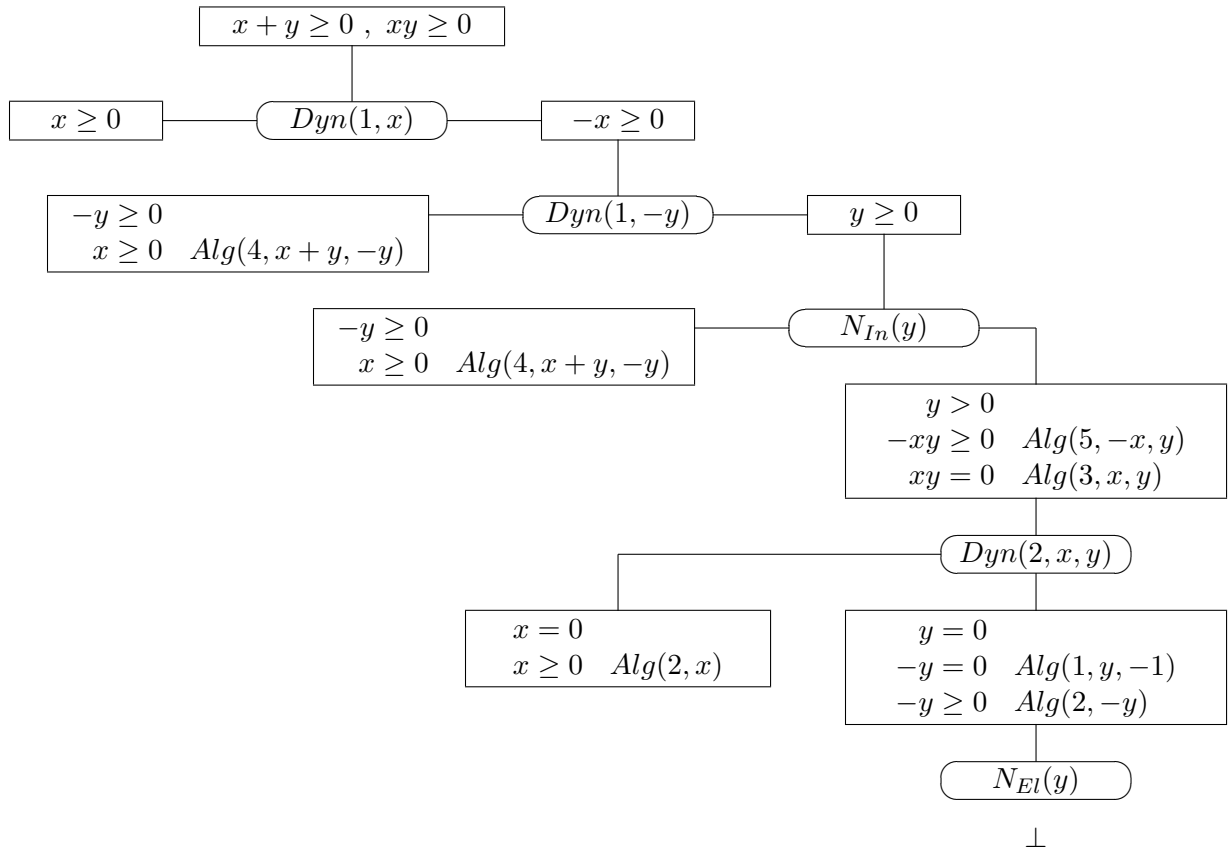
$$\setlength{\unitlength}{1mm}
\begin{picture}(160,115)
\put(25,110){\framebox(40,5){$x+y\geq0\ ,\ xy\geq0$}} \put(45,110){\line(0,-
1){7.5}}
\put(45,100){\oval(25,5)} \put(45,100){\makebox(0,0){$Dyn(1,x)$}}
\put(32.5,100){\line(-1,0){12.5}}
\put(0,97.5){\framebox(20,5){$x\geq0$}}
\put(57.5,100){\line(1,0){12.5}}
\put(70,97.5){\framebox(20,5){$-x\geq0$}} \put(80,97.5){\line(0,-1){7.5}}
\put(80,87.5){\oval(25,5)} \put(80,87.5){\makebox(0,0){$Dyn(1,-y)$}}
\put(67.5,87.5){\line(-1,0){12.5}}
\put(5,80){\framebox(50,10){$\begin{array}{rl} -y\geq0&\\ x\geq0&
Alg(4,x+y,-y)
\end{array}$}}
\put(92.5,87.5){\line(1,0){12.5}}
\put(105,85){\framebox(20,5){$y \ge 0$}} \put(115,85){\line(0,-1){10}}
\put(115,72.5){\oval(25,5)}
\put(115,72.5){\makebox(0,0){$N_{In}(y)$}} \put(102.5,72.5){\line(-1,0){12.5}}
\put(40,65){\framebox(50,10){$\begin{array}{rl} -y\geq0&\\ x\geq0&
Alg(4,x+y,-y)
\end{array}$}}
\put(127.5,72.5){\line(1,0){7.5}}
\put(135,72.5){\line(0,-1){7.5}}
\put(110,50){\framebox(50,15){$\begin{array}{rl} y>0&\\-xy\geq0&Alg(5,-x,y)\\
xy=0&Alg(3,x,y) \end{array}$}} \put(135,50){\line(0,-1){5}}
\put(135,42.5){\oval(25,5)} \put(135,42.5){\makebox(0,0){$Dyn(2,x,y)$}}
\put(122.5,42.5){\line(-1,0){42.5}}
\put(80,42.5){\line(0,-1){7.5}}
\put(60,25){\framebox(40,10){$\begin{array}{rl} x=0&\\ x\geq0& Alg(2,x)
\end{array}$}}
\put(135,40){\line(0,-1){5}}
\put(110,20){\framebox(50,15){$\begin{array}{rl} y=0&\\ -y=0& Alg(1,y,-1)\\ -
y\geq0&Alg(2,-y)
\end{array}$}}
\put(135,20){\line(0,-1){5}}
\put(135,12.5){\oval(25,5)} \put(135,12.5){\makebox(0,0){$N_{El}(y)$}}
\put(135,5){\makebox(0,0){$\bot$}}
\end{picture}$$

It is now possible to suppress the last occurrence of $N_{El}(y)$ and not to
use
$N_{In}(y)$
anymore.

$$\setlength{\unitlength}{1mm}
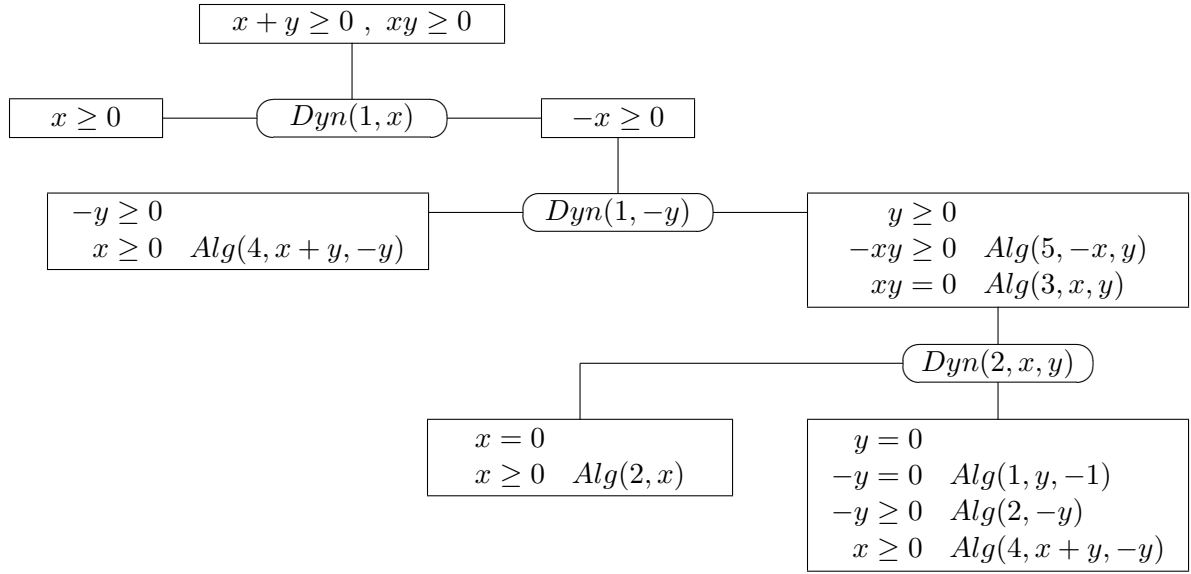
\begin{picture}(160,75)
\put(0,-37.5){\begin{picture}(0,0)
\put(25,110){\framebox(40,5){$x+y\geq0\ ,\ xy\geq0$}} \put(45,110){\line(0,-
1){7.5}}
\put(45,100){\oval(25,5)} \put(45,100){\makebox(0,0){$Dyn(1,x)$}}
\put(32.5,100){\line(-1,0){12.5}}
\put(0,97.5){\framebox(20,5){$x\geq0$}}
\put(57.5,100){\line(1,0){12.5}}
\put(70,97.5){\framebox(20,5){$-x\geq0$}} \put(80,97.5){\line(0,-1){7.5}}
\put(80,87.5){\oval(25,5)} \put(80,87.5){\makebox(0,0){$Dyn(1,-y)$}}
\put(67.5,87.5){\line(-1,0){12.5}}
\put(5,80){\framebox(50,10){$\begin{array}{rl} -y\geq0&\\ x\geq0&
Alg(4,x+y,-y)
\end{array}$}}
\put(92.5,87.5){\line(1,0){12.5}}
\put(105,75){\framebox(50,15){$\begin{array}{rl} y \ge 0&\\-xy\geq0&Alg(5,-
x,y)\\ xy=0&Alg(3,x,y) \end{array}$}} \put(130,75){\line(0,-1){5}}
\put(130,67.5){\oval(25,5)} \put(130,67.5){\makebox(0,0){$Dyn(2,x,y)$}}
\put(117.5,67.5){\line(-1,0){42.5}}
\put(75,67.5){\line(0,-1){7.5}}
\put(55,50){\framebox(40,10){$\begin{array}{rl} x=0&\\ x\geq0& Alg(2,x)
\end{array}$}}
\put(130,65){\line(0,-1){5}}
\put(105,40){\framebox(50,20){$\begin{array}{rl} y=0&\\ -y=0& Alg(1,y,-1)\\ -
y\geq0&Alg(2,-y)\\ x\geq0& Alg(4,x+y,-y) \end{array}$}}
\end{picture}}
\end{picture}$$

Theorem \ref{Theoreme 1.3} can be seen as a	``cut elimination theorem'',
with a	constructive sense: there	is a procedure to transform any proof
	of a fact in a	deduction system for first order logic, into a
dynamical
proof. This seems very closely related to a lemma of Troelstra-Schwichtenberg
(cf. proposition p. 84 in \cite{ST})
concerning intuitionnistic proof systems. A non-constructive proof via topos
theory will be outlined in Subsection \ref{topos}.

\subsection{Collapse}
We consider now dynamical theories with one or several collapse axioms.
Collapse axioms express that a particular fact (or conjunction of facts)
involving only constants cannot be true in a model of $\cd$. For example in an
ordered field the collapse axiom is $$0 >0 \,\vdash\,\bot\;.$$

\begin{definition}
A presentation $(G;R)$ {\em collapses} in the theory $\cd$ when one has
constructed a covering of
$(G;R)$ in $\cd$ where all branches finish with a dead node. Such a
covering is
a dynamical proof of $R\,\vdash\, \bot $ in $\cd$, and will be called a
{\em collapse of $(G,R)$}.
\end{definition}

Remark that a collapse of a presentation $(G;R)$ gives a dynamical proof
of any
fact $B(t)$ in the presentation.

For example the presentation $((x,y),(x^3-y^3=0,(x-y)^2>0))$ collapses in the
theory of ordered fields: take the following dynamical proof
\begin{itemize}
\item
$[0]$ where the generators are $(x,y)$ and the relations are $(x^3- y^3=0, (x-
y)^2>0)$,
\begin{itemize}
\item
$[l]$ where the generators are $(x,y)$ and the relations are $(x^3-y^3=0,(x-
y)^2>0, x=0,x-y=0, 0 > 0)\;$,
so that it is a dead node,
\item
$[r]$ where the generators are $(x,y)$ and the relations are $(x^3- y^3=0,
(x-y)^2>0, x ^2>0,(3x^2/4+(y+x/2)^2)>0)$, \begin{itemize}
\item $[rd]$ where the
generators are $(x,y,z)$ and the relations are $(x^3-y^3=0, (x-y)^2>0, x^2>0,
(3x^2/4+(y+x/2)^2)>0,z(3x^2/4+(y+x/2)^2)-1=0,x-y=0,0 > 0)\;$, so that it is a
dead
node too.
\end{itemize}
\end{itemize}
\end{itemize}

\begin{definition}
Let $\cd$ and $\cd'$ be two dynamical theories with the same language. We say
that $\cd$ and $\cd'$
{\em collapse simultaneously} if for any presentation $(G;R)$ it is
possible to
construct a collapse of $(G,R)$ in $\cd$ from any collapse of $(G,R)$ in
$\cd'$, and vice versa.
\end{definition}

Some dynamical theories that are very different may nevertheless collapse
simultaneously. For example, we are going to prove that the theory of
commutative rings with a proper monoid (see below Section
\ref{subsec-Hilb-simcol}) and the
theory of algebraically closed fields collapse simultaneously.

A stronger connection between dynamical theories is the following:

\begin{definition}
Let $\cd$ and $\cd'$ be two theories with the same
language. The theories $\cd$ and $\cd'$
{\em prove the same facts} if for any presentation $(G;R)$ and any fact $B(t)$
in this presentation, it is possible to construct a dynamical
proof of $R \,\vdash\, B(t) $ in $\cd$ from any dynamical
proof of
$R \,\vdash\, B(t) $ in $\cd'$, and vice versa. \end{definition}

For example, we are going to prove that
the dynamical theories of ordered fields and of real closed fields prove the
same facts, when written in the language of rings with three unary relations
$x=0$, $x>0$ and $x\ge 0$.

\subsection{Dynamical
theories and coherent toposes}\label{topos}

The concept of a dynamical theory has also been known in categorical logic
under
the name of {\it coherent theory}, and it is related to {\it coherent
toposes}.
This subsection is an extended remark to make this connection clear. Some
familiarity with Grothendieck topologies and toposes is useful to read this
subsection.

See for instance \cite{MR} for the
relations between toposes and coherent theories.

Let us consider a dynamical theory $\cd$ in a language $\cl$, and let
$\cd_0$ be
an algebraic subtheory of $\cd$.
We will associate to these data a site consisting of a category with finite
projective limits, equipped with a Grothendieck topology
generated by finite coverings.
The category ${\bf C}(\cd_0)$ depends only on the algebraic subtheory $\cd_0$,
while the topology ${\bf T}(\cd)$ is associated to the extra dynamical
axioms of
$\cd$.\par

The objects of ${\bf C}(\cd_0)$ are finite presentations $(G;R)$ in the
language
$\cl$. A morphism from $(G;R)$ to $(F;Q)$ will be a mapping $\varphi$
from
$F$ to the set of terms of $\cl$ built on $G$, such that for any relation
$A(x_1,\ldots,x_n)$ in $Q$ (the $x_i$'s are in $F$), the fact
$A(\varphi(x_1),\ldots,\varphi(x_n))$ is a consequence of the relations $R$ in
the theory $\cd_0$. This syntactic description of ${\bf C}(\cd_0)$ has
obviously
a semantic counterpart: ${\bf C}(\cd_0)$ is (equivalent to) the dual of the
category of finitely presented models of $\cd_0$. This shows by the way that
${\bf C}(\cd_0)$ has all finite projective limits.\par

It is easy to see that for any morphism $\varphi : (G,R)\to (F,Q)$, there
is an
isomorphism $(F\cup F',Q\cup Q')\to (G,R)$ such that the composition of
$\varphi$ with this isomorphism is the canonical morphism $(F\cup F',Q\cup
Q')\to (F,Q)$. \par

Consider now an extra dynamical axiom of $\cd\;$: $$H(\gx) \,\vdash\,\exists
\gy_1\,A_1(\gx,\gy_1)\vee\cdots\vee\exists \gy_k\,A_k(\gx,\gy_k)\;.$$
To this axiom we associate the finite family of the $k$ obvious arrows in
${\bf
C}(\cd_0)$ with common target $(\gx;H)$ and sources $(\gx\cup \gy_i;H\cup
A_i)$
for $i=1,\ldots,k$. These finite families associated to axioms generate the
coverings of a topology ${\bf T}(\cd)$ on ${\bf C}(\cd_0)$, according to the
following rules:
\begin{enumerate}
\item The identity $f: M\to M$ is a covering of $M$. \item Let $(g_j:N_j\to
N)_{j=1,\ldots,k}$ be a covering of $N$. Let $\varphi: M\to N$ be any
morphism,
and let $$\matrix{ M_j&\buildrel f_j \over\longrightarrow&M\cr
\Big\downarrow&&\Big\downarrow\varphi\cr N_j&\buildrel g_j
\over\longrightarrow&N\ \cr }$$ be cartesian squares for $j=1,\ldots,k$. Then
the family $(f_j:M_j\to M)_{j=1,\ldots,k}$ is a covering of $M$ \item Let
$(f_i:M_i\to M)_{i\in I}$ be a covering of $M$, and for each $i$ let
$(g_{i,j}:N_{i,j}\to M_i)_{j=1,\ldots,k_i}$ be a covering of $M_i$.
Then the family $(f_i\circ g_{i,j}:N_{i,j}\to M)_{i,j}$ is a covering of $M$.
\item Let $(f_i:M_i\to M)_{i\in I}$ be a covering of $M$. If $(g_j:N_j\to
M)_{j\in J}$ is another family such that there is an application $\mu:I\to J$,
and for each $i$ a morphism $\rho_i:M_i\to N_{\mu(i)}$ satisfying
$g_{\mu(i)}\circ\rho_i=f_i$, then $(g_j)_{j\in J}$ is also a covering of $M$.
\end{enumerate}
It is easy to see that, in the generation of coverings for the topology, this
fourth rule can always be used in the last place.

Of course, these rules (at least the first three) parallel the rules of
construction of coverings of a presentation $(G,R)$ in $\cd$. This implies
that
the family $$\Big((G\cup F_j; R\cup Q_j)\longrightarrow
(G;R)\Big)_{j=1,\ldots,k}$$ is a covering for the topology if and only if
there
is a covering of the presentation $(G;R)$ such that
every leaf contains one of the presentations $(G\cup F_j; R\cup Q_j)$, for
$j=1,\ldots,k$.
Stated in another way, the ``sequent''
$$H(\gx) \,\vdash\,\exists \gy_1\,A_1(\gx,\gy_1)\vee\cdots\vee\exists
\gy_k\,A_k(\gx,\gy_k)$$
is a valid dynamical rule in $\cd$ if and only if the family $$\big((\gx\cup
\gy_i;H\cup A_i)\longrightarrow (\gx,H)\big)_{i=1,\ldots,k}$$ is a covering
for
the topology ${\bf T}(\cd)$. \par

Once we have the category ${\bf C}(\cd_0)$ and its topology ${\bf T}(\cd)$, we
can define the sheaves on it. The category of these sheaves is a Grothendieck
topos ${\bf E}(\cd)$, which is known in categorical logic as {\it the
classifying topos} of the theory $\cd$. It is a coherent topos since the
topology ${\bf T}(\cd)$ is generated by finite coverings. There is a canonical
functor $\epsilon : {\bf C}(\cd_0)\to{\bf E}(\cd)$ which sends the object
$M$ of
${\bf C}(\cd_0)$ to the sheaf associated to the presheaf $\hom_{{\bf
C}(\cd_0)}(-,M)$.
A family $(f_i:M_i\to M)_{i\in I}$ is carried by $\epsilon$ to a surjective
family if and only if it is a covering for the Grothendieck topology ${\bf
T}(\cd)$.

It is possible to define what is a model of a coherent (or dynamic) theory
$\cd$
in a topos, and inverse images of geometric morphisms of toposes carry
models of
$\cd$ to models of $\cd$. The classifying topos ${\bf E}(\cd)$ comes equipped
with a model ${\bf M}(\cd)$ of the theory $\cd$, which is {\it generic} in the
following sense: for any model ${\cal M}$ of $\cd$ in any topos ${\cal E}$,
there is a geometric morphism of toposes $f:{\cal E}\to {\bf E}(\cd)$ such
that
${\cal M}$ is isomorphic to $f^*\big({\bf M}(\cd)\big)$. It is easy to
describe
what is the generic model of $\cd$ in the presentation of the classifying
topos
we gave. The assignment, to any presentation $(G;R)$, of the model of $\cd_0$
with this presentation defines a presheaf of models of
$\cd_0$ on ${\bf C}(\cd_0)$. The sheaf associated to this presheaf for the
topology ${\bf T}(\cd)$ is the generic model of $\cd$. In other words, the
generic model is the image by $\epsilon$ of the presentation $(z;\emptyset)$
(where $z$ is one variable).\par

It follows from the interpretation of the language in the generic model of
$\cd$
that a ``sequent''
$$
H(\gx) \,\vdash\,\exists \gy_1\,A_1(\gx,\gy_1)\vee\cdots\vee\exists
\gy_k\,A_k(\gx,\gy_k)
$$
is valid in the generic model if and only if the family of morphisms $$
\big((\gx\cup \gy_i;H\cup A_i)\longrightarrow (\gx,H)\big)_{i=1,\ldots,k} $$
in ${\bf C}(\cd_0)$ is sent by $\epsilon$ to a surjective family, i.e., if and
only if it is a covering for the topology ${\bf T}(\cd)$. By what was said
before, this is equivalent to the fact that the sequent is a valid dynamical
rule in $\cd$. We can then get a non-constructive version of Theorem
\ref{Theoreme 1.3} from a theorem of Deligne asserting that ``a coherent topos
has enough points''. A point of the topos ${\bf E}(\cd)$ is a geometric
morphism
from the topos of sets to ${\bf E}(\cd)$, so it corresponds to a set-theoretic
model of $\cd$. Deligne's theorem says that the ``sequent'' $$H(\gx)
\,\vdash\,\exists \gy_1\,A_1(\gx,\gy_1)\vee\cdots\vee\exists
\gy_k\,A_k(\gx,\gy_k)$$
is valid in the generic model of $\cd$ if and only if it is valid in any set-
theoretic model of $\cd$. In conclusion, the sequents valid in every (set-
theoretic) model of $\cd$ are exactly the valid
dynamical rules. \par
A collapse axiom in a dynamical theory $\cd$ gives, by the construction of
the
topology ${\bf T}(\cd)$, an empty covering of a subobject $U$ of the terminal
object $(\emptyset;\emptyset)$ in ${\bf C}(\cd_0)$. If $\cd$ consists of
$\cd_0$
plus a collapse axiom, then the classifying topos for $\cd$ is a closed
subtopos of the topos of presheaves ${\bf C}(\cd_0)\hat{\ }$, complement to
$\epsilon(U)$.\par In the presentation of dynamical proofs and collapses, we
have considered possibly infinite presentations $(G;R)$ to start with. To deal
with this situation, one may add to the language the generators in $G$ as new
constants and the relations in $R$ as new axioms to $\cd$ (or $\cd_0$). We
denote by $(G;R)/\cd$ the dynamical
theory thus obtained (whose models are those of $\cd$ ``under'' $(G;R)\;$). We
can then construct a site as above, and a classifying topos ${\bf
E}\big((G;R)/\cd\big)$.\par Now we can take into account, in this topos-
theoretic setting, the non-constructive aspects of the collapse of a
presentation $(G;R)$ in the theory $\cd\;$: it collapses if and only if the
classifying topos ${\bf E}\big((G;R)/\cd\big)$ is the trivial topos, where the
initial object is also terminal. In case that the presentation $(G;R)$ is
finite, it means also that the object $(G;R)$ has an empty covering
in the topology ${\bf T}(\cd)$, or equivalently that its image in the
classifying topos of $\cd$ is the initial object $0$.\par
It is
also possible to describe the simultaneous collapsing along the same lines.
For simplicity we shall
consider two dynamical theories $\cd$ and $\cd'$ in the same language, with
$\cd$ a subtheory of $\cd'$. This gives a geometric morphism $f: {\bf
E}(\cd')\to {\bf E}(\cd)$ (${\bf T}(\cd')$ is finer than ${\bf T}(\cd)$).
Forgetting the constructive aspects, we get:

\begin{proposition} The theories $\cd$ and $\cd'$ collapse simultaneously
if and
only if, for every object $X$ of ${\bf E}(\cd)$,
$f^*(X)=0$ implies $X=0$ (i.e., $f^*$ reflects the initial object).
\end{proposition}

So the simultaneous collapsing is in some sense independent of the syntax,
since
it can be formulated only in terms of the classifying toposes. In the topos-
theoretic framework, the ``syntax'' means the choice of the site ${\bf
C}({\cal
D}_0)$ defining the topos ${\bf E}({\cal D})$ (or more precisely its image in
${\bf E}({\cal D})$ by the functor $\varepsilon$). In this sense, the stronger
relation of ``proving the same facts'' depends on the syntax. Let us take
an ad-
hoc example. Consider the theories of commutative rings with a proper
multiplicative monoid whose elements are not zero divisors (resp. are
invertible). If these two theories are formulated in the language with one
unary
relation symbol for the monoid, they prove the same facts: indeed, the
morphism
from a ring $A$ to its ring of fractions $M^{-1}A$ is injective if the monoid
$M$ contains no zero divisor. On the other hand, if one adds another unary
relation symbol interpreted as ``being invertible'', the two theories no
longer
prove the same facts.

Two theories $\cd$ and $\cd'$
as in the proposition prove the same facts if and only if every monomorphism in
${\bf C}(\cd_0)$ which becomes an isomorphism in ${\bf E}(\cd')$ already
becomes
an isomorphism in ${\bf E}(\cd)$. If every monomorphism of ${\bf C}(\cd_0)$
becomes complemented in ${\bf E}(\cd)$, then simultaneous collapsing implies
proving the same facts.

\section{Hilbert's Nullstellensatz}

\subsection{Direct theories}

We begin this section by a discussion about the theory of rings.

The {\em unary language of rings} ${\cal L}_r$
has constants $0$, $1$, $-1$ and binary
functions
$+$ and $\times$ and only one unary relational symbol $=0$. As usual, $x\times
y$ will often be
denoted by $xy$, $-t$ will
stand for $(-1)\times t$ and $s-t$ for $s+(-t)$.

The {\em theory of rings} i.e., the
purely equational theory of commutative rings, expressed in this language will
be denoted ${\cal R}_r$.

In this setting, a presentation in the language is nothing but a set of
variables $G$ and a set of polynomials $\Rzero \subset \zg $,
with relations
$p(G)=0$ for $p(G) \in \Rzero $.
We denote it by $(G;\Rzero)$.

We see immediately that terms provably $=0$ in ${\cal R}_r$ (for the
presentation we consider) are just the polynomials belonging to the ideal of
$\zg$ generated by $\Rzero$.

In other words:
\begin{itemize}
\item
we manipulate polynomials in $\zg$ rather than terms of the language $\cl_r$,
\item addition and multiplication are directly defined as operations on
polynomials (this hides logical axioms of rings behind algebraic
computations in
$\zg$),
\item the only relation is the unary relation $x=0$, \item we do not have the
binary equality relation, $x=y$ is only an abbreviation for
$x-y=0$,
\item the only axioms are three very simple algebraic axioms:
$$\begin{array}{rlccc}
&\vdash 0= 0&\qquad&\qquad&D(1)_r \\
x = 0,\ y=0&\vdash x+y = 0&\qquad&\qquad&D(2)_r \\ x = 0 &\vdash xy=
0&\qquad&\qquad& D(3)_r \\ \end{array}$$
\end{itemize}

This reformulation of the theory of rings is exactly what we need for
Nullstellens\"atze as we shall see soon.

\medskip

This leads us to the notion of {\em direct theory}.

A {\em direct algebraic axiom} is an axiom of the form
$$A_1(x_1),\cdots,A_k(x_k) \,\vdash\, A(t(x_1,\ldots ,x_k))$$ where the $A_i$
and $A$ are unary relation symbols, the $x_i$ are distinct variables and
$t(x_1,\ldots ,x_k)$ is a term of the language.

For example the axioms
$$x>0,\;y>0\,\vdash\, x+y>0\qquad {\rm and}\qquad \vdash\, x^2 \ge 0\; $$ are
direct algebraic axioms, while
$$x\ge 0,\ x\not= 0 \,\vdash\, x>0\qquad {\rm and}\qquad x^2 > 0 \,\vdash\,
x\not=
0$$
are not direct algebraic axioms: the first one because $x$ appears twice on
the
left, the second because $x^2$ is not a variable.

Now we say that a purely equational theory {\em is put in unary form} when we
have replaced syntactical terms by objects of free algebraic structures
(free w.
r. t. equational axioms), binary equality relation by a unary one (the old
binary equality with a fixed constant in right-hand side), and equational
axioms by two ingredients: computations in the free algebraic structure on the one
hand and some direct algebraic axioms on the other hand (as we did for theory of
rings). A {\em simple collapse axiom } is a collapse axiom of the form 
$$A(c)\,\vdash\, \bot $$ 
where $A$ is a unary relation symbol and $c$ is a constant

A {\em direct theory} is a dynamical theory based on a purely equational theory
put in unary form, allowing as other axioms only direct algebraic axioms and exactly one simple collapse axiom.

We now write down the axioms of non-trivial rings.

A {\em non-trivial ring} is a ring where $1=0$ is impossible. The corresponding
theory, expressed in the language $\cl_r$, is the direct theory extending
${\cal R}_r$ by adding only a simple collapse axiom:
$$\begin{array}{rlccc}
1= 0 \ &\vdash \ \bot&\qquad&\qquad&C_r
\end{array}
$$

\begin{proposition} \label{Proposition 1.1bis} Let $(G;\Rzero )$ be a
presentation in the language of rings. A collapse
of the presentation $(G;\Rzero )$ in the theory of non-trivial rings
produces an equality $1 \= a_1i_1+\cdots+a_ki_k$ in $\zg $ with $i_j$ in $\Rzero$.
Reciprocally, such an equality produces a collapse of $(G;\Rzero )$.
\end{proposition}

\begin{proof}{Proof:}
In a direct theory, such as the theory of non-trivial rings, the only
dynamical
axiom is the axiom of collapse. So proofs have a very simple structure
``without branches". The elements of $\zg$ which are provable $= 0$ without
using the collapse axiom (in the presentation $(G;\Rzero )$) are exactly
elements of the form $a_1i_1+\cdots+a_ki_k$
with $i_j$ in $\Rzero$.
This is clear by induction on the number of times the direct algebraic axioms
are used in the proof. We can apply the collapse axiom only after such a
proof
of $1 = 0$. So the presentation collapses in the theory of non-trivial
rings if
and only if there exists an algebraic identity $1- (a_1i_1+\cdots+a_ki_k)\= 0$
in $\zg $ with $i_j$ in $\Rzero$.
\end{proof}


\subsection{Some simultaneous collapses} \label{subsec-Hilb-simcol}

We consider the {\em unary language of fields} $\cl_f$, which is the unary
language of rings with a new unary relation
$\not= 0$.
A presentation in the language $\cl_f$ consists of two sets $(\Rzero,\Rnz
)$ of
polynomials in $\zg $ with relations $p(G)=0$ for $p(G) \in \Rzero $ and $p(G)
\not= 0$ for $p(G) \in \Rnz $. We denote it by $(G;\Rzero,\Rnz )$.

A {\em ring with a proper monoid}
is a ring with a multiplicative monoid not containing 0.
It is the same as a realization of the language $\cl_f$ satisfying the
axioms of
rings and the following axioms
$$\begin{array}{rlccc}
x = 0 ,\ y \not= 0 &\,\vdash\, x+y \not= 0&\qquad&\qquad&D(1)_f\\ x \not= 0 ,\
y\not= 0 &\,\vdash\, xy \not= 0&\qquad&\qquad&D(2)_f \\ &\,\vdash\, \ 1
\not= 0
&\qquad&\qquad&D(3)_f\\ 0\not= 0 &\,\vdash\, \bot &\qquad&\qquad& C_f \\
\end{array}$$
where the proper monoid is the realization of the unary relation $\not=0$.
Note
that a non-trivial ring is a ring with a proper monoid, the proper monoid
being
$\{1\}$.

Direct algebraic axioms are denoted by $D$ and collapse axioms by $C$. Remark
that axiom $D(1)_f$ is a disguised axiom of stability of the relation
$\not= 0$ for equality, written using only unary predicates. Considering axiom
$D(3)_f$ we see that the collapse axiom of non-trivial rings is a valid
dynamical rule in rings with proper monoid.

Adding three extra axioms we get the axioms of the theory of {\em fields}:
 $$\begin{array}{rlccc}
xy-1=0&\,\vdash\, x \not= 0
&\qquad&\qquad&S(1)_f
\\ x
\not= 0&\,\vdash\,
\exists y\ xy -1 =0&\qquad&\qquad&Dy(1)_f \\ &\,\vdash\, x = 0 \ \vee \ x
\not=
0&\qquad&\qquad& Dy(2)_f\\ \end{array}$$

The first axiom is a {\em simplification axiom}: an algebraic axiom but not a
direct algebraic one.
The two last ones are dynamical axioms.

The theory of {\em algebraically closed fields} is obtained by adding a scheme
of axioms. For every degree $n$ we have the axiom: $$\,\vdash\, \exists y \
y^n
+ x_{n-1} y^{n-1}+ \cdots+ x_1 y + x_0 = 0 \ \ Dy_n(3)_f$$

The theory of rings with proper monoid has been chosen because as we shall see
later
it is a
direct theory which collapses simultaneously with the theory of algebraically
closed fields.

The collapse in the theory of rings with proper monoid has a very
simple form:

\begin{proposition} \label{Lemme 2.3}
Let ${\cal K}=(G;\Rzero,\Rnz )$ be a presentation in the language $\cl_f$. A
collapse of the presentation ${\cal K}$ in the theory of rings with proper
monoid produces an equality in $\zg\;$: $$m_1\cdots m_\ell+a_1 i_1+\cdots
+a_ki_k\= 0$$ with $m_j$ in $\Rnz$ and $i_j$ in $\Rzero$. Reciprocally,
such an
equality produces a collapse of ${\cal K}$. \end{proposition}
\begin{proof}{Proof:}
First consider dynamical proofs of facts using {\em only direct algebraic
axioms}. These are algebraic proofs without branching.
The elements of $\zg$ which are provable $= 0$ in the presentation $(G;\Rzero
)$) are exactly elements of the form $a_1i_1+\cdots+a_ki_k$ with $i_j$ in
$\Rzero $. This is clear by induction on the number of times the direct
algebraic axioms are used in the proof. Then provably $\not= 0$ elements are
exactly elements of the form $m_1\cdots m_\ell+a_1 i_1+\cdots +a_k i_k$ with
$m_j \in \Rnz$ and $i_j \in\Rzero$ (same inductive proof).

Now a proof of collapse is given by a proof of $0\not= 0$ using only direct
algebraic axioms. It produces an equality $m_1\cdots m_\ell+a_1 i_1+\cdots
+a_ki_k \= 0$ in $\zg $ with $m_j$ in $\Rnz$ and $i_j$ in $\Rzero$.
\end{proof}

The content of the preceding proposition is that the collapse of a
presentation
in the direct theory we consider may be certified by an algebraic identity of
some type. We will try in the following remark to analyze the ingredients we
used to
establish this property, and we will check in the other sections that these
ingredients are again at work.\par

\begin{remark}\label{ordre-des-predicats} {\rm There is an ordering on the
predicates, which appears in the proof of Proposition \ref{Lemme 2.3}. First
comes $=0$, then $\not=0$. This appears also in the syntactic description
of the
theory of rings with proper monoid. We can see the axioms $D(i)_r$ as {\em
construction axioms} for $=0$, and the axioms $D(i)_f$ as construction axioms
for $\not=0$. The rule is that, in a construction axiom for a predicate $Q$
(with $Q$ appearing at the right side of $\vdash$), another predicate $P$ may
appear at the left of $\vdash$ only if $P$ precedes $Q$. The collapse of the
direct theory involves the last predicate. It appears that, when this
scheme is
present, a collapse of a presentation in the direct theory produces an
algebraic identity of a certain type, certifying the collapse.} \end{remark}

So we have algebraic identities certifying collapses in a direct theory. We
are
not interested in this theory, but in some of its extensions. It remains to
obtain a result of simultaneous collapsing.

\begin{theorem}\label{th-coll-rings-fields} The theory of rings with
a proper monoid, the theory of fields and the theory of algebraically closed
fields collapse simultaneously.
\end{theorem}

\begin{proof}{Proof:}
The proof is by induction on the number of times the axioms of algebraically
closed fields $S(1)_f,\ Dy(1)_f,\ Dy(2)_f$ and $Dy_n(3)_f$ are used in the
proof. We have to see that if after one use of such an axiom we get a
collapse
of the new presentations in the theory of rings with a proper monoid then
we can
also get the collapse of the preceding presentation in the same theory.

Thus the theorem is an immediate consequence of the following lemma.
\end{proof}

Before stating and proving the lemma, we introduce some conventional abuse of
notations to be used when the context is clear.

\begin{notation} \label{nota-abus}
{\rm Assume we have a presentation ${\cal K}=(G;\Rzero,\Rnz )$, $z$ is a new
variable, $p$, $q$ are in $\zg $, $r(z)$ and $s(z)$ are in $\zg [z]$, then the
presentation $(G\cup \{z\};\Rzero \cup \{p,r(z) \},\Rnz \cup \{q,s(z) \} )$
will
be denoted by ${\cal K}\cup (p=0,r(z)=0,q\not= 0,s(z)\not= 0)$ }
\end{notation}

\begin{lemma} \label{Lemme 2.4}
Let ${\cal K}=(G;\Rzero,\Rnz )$ be a presentation in the language $\cl_f$.
Let $p,r \in \zg $. Let $z$ be
a new variable and $q(z)$ a monic non-constant polynomial in $\zg [z]$.

a) If the presentation
${\cal K} \cup ( p\not= 0 )$
collapses in the theory of rings with proper monoid, so does the
presentation ${\cal K} \cup
( pr-1=0 )$

b) If the presentation
${\cal K} \cup ( pz-1=0 )$
collapses in the theory of rings with proper monoid, so does the
presentation ${\cal K} \cup
(p\not= 0 )$

c) If the presentations
${\cal K} \cup (p=0 )$
and
${\cal K} \cup (p\not= 0)$
collapse in the theory of rings with proper monoid, so does the
presentation ${\cal K}$.

d) If the presentation
${\cal K} \cup (q(z)=0)$
collapses in the theory of rings with proper monoid, so does the presentation
${\cal K}$.
\end{lemma}

\begin{proof}{Proof:} Denote by $\Izero$ the ideal of $\zg$ generated by
$\Rzero$ and by $\Mnz$ the monoid generated by $\Rnz$.

a) We have an identity $p^n m\= i$ with $m$ is in $\Mnz$ and $i\in \Izero$. We
can multiply it by $r^n$. We can write
$1-(pr)^n\= (pr-1)s$ so that
$m\= (1-(pr)^n) m+ i r^n \= (pr-1) sm+ i r^n$, which produces a
collapse of ${\cal K} \cup (pr-1=0)$

b)
This is Rabinovitch's trick.
Suppose we have a collapse of the presentation ${\cal K} \cup ( pz-1=0 )$.
This is written in the form
$\ m\= j(z) + (pz - 1) s(z)\;$,
where $m$ is in $\Mnz$, $j$ is a polynomial with coefficients in $\Izero$ and
$s$ is a polynomial with coefficients in $\zg $. If $n$ is the $z$-degree of
$j$, multiply both sides by $p^n$ and replace in $p^n j(z)$ all the
$p^kz^k$ by
$1$ modulo $(pz - 1)$. After this transformation, we obtain an equality $p^n
m \=
i+(pz-1)s_1(z)$ in $\Z[G,z]$, where
$i\in
\Izero$. We can assume that $p$ is not $0\in\zg$, otherwise ${\cal K}\cup
(p\neq
0)$ collapses trivially. It follows that $p^nm\= i$, which is the collapse we
are looking for.

c) Since the presentation
${\cal K} \cup (p\not= 0 )$ collapses,
we have an equality $\ \ m p^n \= i\ \ $ in $\zg $ with $ m\in \Mnz$ and $i\in
\Izero$.
Similarly we have an equality $\ \ v \= i' + p a\ \ $ in $\zg $ with $v\in
\Mnz$
and $i'\in \Izero$. So we get equalities in $\zg $ $$ i a^n \= m p^n a^n \=
m (v
- i')^n \= m v^n + i_2$$ with $i_2\in \Izero$,
which gives
$m v^n + i_3 \= 0 $ with $m v^n\in \Mnz$ and $i_3 \= i_2 - i a^n\in \Izero$.

d) We suppose that there is an equality in $\Z[G,z]$ of the form $$m + \sum _i
r_i a_i(z) + q(z) a(z) \= 0 \eqno{(1)}$$ where the $r_i$ belong to $\Rzero$
and
$m$ belongs to $\Mnz$, the monoid generated by $\Rnz$.
Dividing the $a_i$ by the monic polynomial $q$, we get an algebraic
identity $$
m+ \sum _i r_i b_i(z) + q(z) b(z) \= 0 \eqno{(2)}$$ where the $b_i$ are of
degree smaller that $\deg(q)$ in $z$. So $b(z)\= 0$ and
$$m+ \sum _i r_i b_i(0) \= 0 \eqno{(3)}$$ which is the collapse we are
looking
for.
\end{proof}

\begin{remark}
{\rm This kind of result is particularly easy because we have stated an
algebraic
form of collapse (in Proposition \ref{Lemme 2.3}). The algebraic computation
constructing a collapse from several other ones is in fact present (and often
hidden) in classical proofs of ``embedding theorems" as ``every proper
ideal is
contained in a prime ideal" or ``every field is embeddable in an algebraically
closed field".\par The proofs of simultaneous collapsing that we will
encounter
in this paper will always
use this technique of lifting algebraic identities certifying the collapses
along the extra dynamical axioms. \par
All examples we give in this paper present simultaneous collapsing between a
dynamical theory and some direct subtheory. It would be interesting to have
general criteria for such a simultaneous collapsing. }\end{remark}

As an immediate consequence of Theorem \ref{th-coll-rings-fields} we get:
\begin{corollary}
\label{Corollaire 2.7} (constructive versions of non-constructive embedding
theorems)
Let $A$ be a ring. If the diagram of $A$ collapses in the theory of
algebraically closed fields, then $A$ is trivial. In particular we get:

a) Let $A$ be a non-trivial ring. The diagram of $A$ does not collapse in the
theory of fields.

b) Let $K$ be a field. The diagram of $K$ does not collapse in the theory of
algebraically closed fields. \end{corollary}

In this proposition, the claim a) is a constructive version of the following
result: ``if a ring is non-trivial, it has a prime ideal".

In the same way, the claim b) is a constructive version of the fact that
``every
field can be embedded in an algebraically closed field".

Constructive versions of embedding results similar to the ones stated above
are
announced in a note by Joyal \cite{Joyal}. They rely on a lattice-theoretic
description of the spectrum of a ring.

Theorem \ref{th-coll-rings-fields}
can also be settled in the following form.
\begin{proposition} \label{Proposition 2.8bis} Let ${\cal
K}=(G;\Rzero,\Rnz)$ be
a presentation in the language $\cl_f$.
A collapse of the presentation ${\cal K}$ in the theory of algebraically
closed
fields produces an equality $m_1\cdots m_\ell+a_1 i_1+\cdots +a_ki_k=0$ with
$m_j$ in $\Rnz$ and $i_j$ in $\Rzero$. \end{proposition}

We can deduce from this last result a
non-constructive formal version of Hilbert Nullstellensatz.

\begin{proposition} \label{Proposition 2.9} Let $A$ be a ring, and $\Rzero$
and
$\Rnz$
families of elements of
$A$. Denote by $\Izero$ the ideal generated by $\Rzero$ and by $\Mnz$ the
monoid
generated by $\Rnz$.
The following properties are equivalent:

i) There exist $i \in \Izero$ and $m \in \Mnz$ with $ i + m= 0$

ii) There exists no homomorphism $\phi : A \rightarrow L$ with $L$ an
algebraically closed field,
$\phi(i)=0$ for $i \in \Rzero$
and $\phi(m)\not= 0$ for $m \in \Rnz$.

iii) There exists no prime ideal $I$ containing $\Izero$ and not intersecting
$\Mnz$.
\end{proposition}

\begin{proof}{Proof:}
Use the preceding result taking as presentation $\dg(A) \cup (\emptyset;
\Rzero,
\Rnz)$, and apply the non-constructive completeness theorem of model theory.
\end{proof}

\subsection{Decision algorithm and constructive Nullstellensatz} \label{subsec-eqtf-Hilb-Nst}

Since the theory of algebraically closed fields has a decision algorithm for
determining emptiness of sets defined by equations and negations of equations
which is particularly simple, we are able to prove the following:

\begin{proposition}\label{prop-dyna-elim} Let $K$ be a field and $\Rzero,\
\Rnz$
two finite families of
polynomials of $K[x_1,\ldots,x_n]$.
There is a decision algorithm answering yes or no to the question ``does the
presentation
$\dg(K)\cup (\{x_1,\ldots,x_n \} ;\Rzero,\Rnz)$ implies $\bot$ in the theory of
algebraically closed fields ?"
If the answer is yes, the
algorithm produces a collapse of the presentation in the theory of fields.
\end{proposition}

\begin{proof}{Proof:} We assume that,
from a constructive point of view, all our fields are discrete. This means
that
we have a way of deciding exactly if an element is zero or not. Precisely,
$G_1$
being the finite set of coefficients of polynomials belonging to $\Rzero \cup
\Rnz$, we can decide for any $\Z$-polynomial whether it vanishes
or not when evaluated on $G_1$ in $K$.

We give a sketch of an elementary decision algorithm, very near to dynamical
evaluation in the dynamical constructible closure of a field
(see \cite{Gom93}).

We deal first with only one variable $x$, and we show that any
finite set of constraints
$(p_i(x)=0)_{1\le i\le h},\ (q_j(x)\not= 0)_{1\le j\le k}$ ($h$ and $k$ are
natural integers) is equivalent to only one constraint. Moreover the
equivalence
is provable by a dynamical proof within the theory of fields.

If $h>0$, the constraints $(p_i(x)=0)_{1\le i\le h}$ are equivalent to a
single
one $p(x)=0$ where $p$ is some gcd of $p_i$'s. We remark that the
computation of
$p$ by Euclid's algorithm
is a computation in the fraction field of the ring $\Z[G_1]\subset K$. Using
pseudo-remainders instead of remainders we have a computation within $\Z[G_1]$.
We get a Bezout relation $\gamma p=a_1p_1+\ldots a_hp_h$, and divisibility
relations $\alpha_ip_i=b_ip$ (greek letters mean non-zero elements of $K$). So
dynamical proofs of $$\dg(K),\, p=0 \,\vdash\,(p_1=0,\ldots,\,p_h=0)\qquad
{\rm
and}$$ $$ \dg(K),\, p_1=0,\ldots, \,p_h=0 \,\vdash\, p=0$$ are very easy.

If $k>0$, the constraints $(q_j(x)\not= 0)_{1\le j\le k}$ are equivalent to a
single one,
$q\not= 0$ where $q=q_1\ldots q_k$. A dynamical proof for this equivalence is
also very easy.

Finally, we have to see the case of a system of two constraints $(p=0,\ q\not=
0)$.
If $q=0$ or $p=0$ in $K[x]$, the system is equivalent to $q\not=0$. Else,
we can
compute within the subring $\Z[G_1]$ the part of $p$ prime to $q$.
More precisely we get some equalities involving polynomials in $\Z[G_1]\;$:
$\beta p=p_1p_2,\ p_1u+qv=\alpha,\ p_{2}q_{2}=\gamma q^k\;$. From these
equalities we get dynamical proofs of $$\dg(K),\, p=0,\ q\not= 0 \,\vdash\,
p_1=0 \qquad {\rm and}\qquad \dg(K),\, p_1=0 \,\vdash\, (p=0,\ q\not= 0)\;.$$

So we can always reduce the problem to only one constraint. After this
reduction,
we get a collapse if we obtain
as constraint $t=0$ with a non-zero constant $t$ of $K$, or a constraint
$q\not=
0$ with $q=0$ in the case that $h=0$ (this means that one $q_j$ is actually
0).

We have to see that in the other cases the collapse is impossible.

In the case of only one constraint $t(x)=0$ with $deg(t)>0$,
we may assume that $t(x)$ is monic.
Now Lemma \ref{Lemme 2.4} d) implies that if $(\dg(K),\ t(x)=0)$ collapses,
then $\dg(K)$ collapses also,
which is impossible.

In the case of only one constraint $q(x)\not= 0$, the constraint is true in
the
field
$K(x)$ of rational fractions.

This ends the proof of the one variable case. For the general case we need the
following lemma.

\begin{lemma} \label{Lemme 2.4bis}
Let ${\cal K}=(G;\Rzero,\Rnz )$ be a presentation in the language $\cl_f$.
Let $z$ a new variable and $q(z)$
a monic non-constant
polynomial in $\zg [z]$.
Let $q_1(z)$ be a polynomial $\zg [z]$ with leading coefficient $p$.

i) If the presentation
${\cal K} \cup (q(z)\neq0)$
collapses in the theory of rings with proper monoid, so does the presentation
${\cal K}$.

ii) If the presentation
${\cal K} \cup (q_1(z)=0, p\neq 0)$
collapses in the theory of rings with proper monoid, so does the presentation
${\cal K}\cup (p\neq0)$.

iii) If the presentation
${\cal K} \cup (q_1(z)\neq 0, p\neq 0)$
collapses in the theory of rings with proper monoid, so does the presentation
${\cal K}\cup (p\neq0)$.
\end{lemma}

\begin{proof}{Proof:}
i) We suppose that there is an
equality in $\Z[G,z]$ of the form $$mq(z)^n+ \sum _i r_i a_i(z) = 0 $$
where the
$r_i$ belong to $\Rzero$ and $m$ belongs to $\Mnz$, the monoid generated by
$\Rnz$. Let $n'$ be the $z$-degree of $q$. This equality in $\Z[G,z]=\zg [z]$
gives for the coefficient of $z^{nn'}$ in $\zg $
exactly an equality in the form of the collapse we are looking for.

ii) We get the result by combination of items a), b) and d) of Lemma
\ref{Lemme 2.4}.

iii) We get the result by combination of i) and of items a), b) in Lemma
\ref{Lemme 2.4}.
\end{proof}

\smallskip
We now turn to the multivariate case.
Let us call $S$ our system of polynomial constraints. We consider the variables
$x_1,\ldots,x_{n-1}$ as parameters and the variable $x_n$ as our true
variable. We try to make the same computations as in the one variable case.
Computations are essentially pseudo-remainder computations. With coefficients
depending on parameters, such a computation splits in many cases, depending on
the degrees of the polynomials, i.e., depending on the nullity or
nonnullity of
polynomials in the parameters. This gives a (very big) finite tree, which is
precisely a covering of the
presentation
$\dg(K)\cup (\{x_1,\ldots,x_n \} ;\Rzero,\Rnz)$ in the theory of fields.
At each leaf $L$ of this tree, we have a presentation with a system $S_{L}$ of
polynomial constraints on $x_1,\ldots,x_{n-1}$ and only one constraint $s_{L}$
on $x_n$
which is either $p_{L}=0$ or $p_{L}\not=0$, where $p_L$ is an $x_n$-polynomial
with coefficients in $K[x_1,\ldots,x_{n-1}]$.

If $p_{L}$ is not a ``constant" (i.e., an element of $K[x_1,\ldots,x_{n-1}]$),
the fact that the leading $x_n$-coefficient of $p_{L}$ is $\neq 0$ is given
by a
polynomial constraint in $S_{L}$. Moreover, in each case we have
dynamical proofs for $S_{L},s_{L}\,\vdash \, S$ and $S_{L},S\,\vdash s_{L}$.

If $p_{L}$ is a ``constant" the system $S'_{L}=(S_{L},s_{L})$ does not involve
$x_{n}$ and the presentation $((x_1,\ldots,x_n),(S_{L},S))$ collapses if and
only if the presentation $((x_1,\ldots,x_{n-1}),S'_{L})$ collapses.

If $p_{L}$ is not a ``constant", by Lemma \ref{Lemme 2.4bis} ii) and iii), the
presentation $((x_1,\ldots,x_n),(S_{L},s_{L}))$ collapses if and only if the
presentation $((x_1,\ldots,x_{n-1}),S_{L})$ collapses. So, the presentation
$((x_1,\ldots,x_n),(S_{L},S))$ collapses if and only if the presentation
$((x_1,\ldots,x_{n-1}),S_{L})$ collapses.

Finally, $S$ collapses iff all the presentations
$((x_1,\ldots,x_n),(S_{L},S))$
at the leaves of the big tree collapse. So we can finish the proof arguing by
induction.

\end{proof}

\begin{theorem}
\label{Theoreme 2.10}
(constructive version of Hilbert's Nullstellensatz) Let $K$ be a field and
$\Rzero,\ \Rnz$
two finite families of
polynomials of $K[x_1,\ldots,x_n]$.
There is an algorithm deciding if the presentation $\dg(K)\cup
(\{x_1,\ldots,x_n\};\Rzero,\Rnz)$ collapses in the theory of algebraically
closed fields. In case of positive answer one can produce an equality $m =
a_1p_1+\cdots+a_kp_k$ with $p_j$ in $\Rzero$ and $m$ in the monoid $\Mnz$
generated by $\Rnz$. \end{theorem}

\begin{proof}{Proof:}
We use Proposition \ref{prop-dyna-elim}, Theorem \ref{th-coll-rings-fields}
saying that the theory of fields collapses
simultaneously with the theory of rings with proper monoids, and finally
Proposition \ref{Lemme 2.3} describing collapse in rings with proper monoids.
\end{proof}

In general, the algebraic closure of a field cannot be constructed. But in
several important particular cases, for example if the field $K$ is
discrete and
enumerable, or discrete and ordered, the algebraic closure can be constructed
(see \cite{MRR88} and \cite{LR90}.). The effective Hilbert's
Nullstellensatz has
a nicer formulation then.
\begin{theorem}
\label{Theoreme 2.10bis}
Let $K$ be a field contained in an algebraically closed field $L$. Let
$\Rzero $
be a finite family of polynomials of $K[x_1,\ldots,x_n]$. One can decide
whether a polynomial $q\in K[x_1,\ldots,x_n]$ is $0$ on the common zeroes of
polynomials $p_j\in \Rzero$ in $L^n$. In case of positive answer one can
produce
an equality $q^n=a_1p_1+\cdots+a_kp_k$, with $p_j$ in $\Rzero$. In case of
negative answer, one can produce
a point in $L^n$ which is a zero of all $p_j$ in $\Rzero$ and not of $q$.
\end{theorem}
\begin{proof}{Proof:}
Consider the algorithm in the proof of Proposition \ref{prop-dyna-elim} with
$\Rnz=\{q\}$.
If we are in the situation where the decision algorithm answers ``yes", we
conclude by the preceding theorem.
In the other case, we consider a leaf of the covering built by
induction on the number of variables in the proof of Proposition
\ref{prop-dyna-elim}, such that the
``triangular system'' at this leaf does not collapse: this system contains
only
one constraint $r_i(x_1,\ldots,x_i)=0$ or $\not=0$ for each $i$,
where $\deg_{x_i}(r_i)>0$ in the case $=0$ and the constraints involving
$x_1,\ldots,x_{i-1}$ imply that the leading coefficient of $r_i$ with
respect to
$x_i$ is $\neq 0$. So the construction of a point satisfying this triangular
system is easy.
\end{proof}

So the constructive character of Hilbert Nullstellensatz comes in our approach
from two different ingredients:
\begin{itemize}
\item the fact that,
when the decision algorithm produces a proof of $\bot$
in the theory of
algebraically closed fields,
the presentation collapses in
the theory of fields,
\item the fact that a collapse in the theory of fields gives rise to a
construction of an algebraic identity certifying this collapse. \end{itemize}

 \subsection{Provable facts
and algebraic theory of quasi-domains} \label{subsec-provfacts-predom}

We give now the axioms of the theory of {\em quasi-domains}: the axioms of
rings
with a proper monoid and the following {\em simplification axioms}.
 $$\begin{array}{rlccc}
x^2 = 0 &\,\vdash\, x = 0 &\qquad&\qquad&S(2)_f \\ xy= 0,\ x \not=
0&\,\vdash\,
y= 0&\qquad&\qquad&S(3)_f\\ xy\not= 0&\,\vdash\, x\not=
0&\qquad&\qquad&S(4)_f\\
\end{array}$$

Remark that the
simplification axiom $S(1)_f$ is a valid dynamical rule
in the theory of quasi-domains.

Note also that a field is a quasi-domain. More precisely, axioms of
quasi-domains
are axioms of fields or valid dynamical rules in the theory of fields. So
quasi-domains are between rings with proper monoid and fields, and we get the
following lemma.

 \begin{lemma}
\label{lem-coll-sim-predom-fields} The theories of rings with proper monoid,
quasi-domains, fields and algebraically closed fields collapse simultaneously.
\end{lemma}

\begin{proposition}
\label{Proposition 2.11}
The theories of quasi-domains, fields and algebraically closed fields prove
the
same facts.
\end{proposition}
\begin{proof}{Proof:}
It is easy to see that in the theory of fields a fact is provable (from a
presentation) if and only if the ``opposite" fact (obtained by replacing $=0$
with $\neq 0$ and vice-versa)
produces a
collapse (when added to the presentation). This is because we have the axiom
$\vdash x=0 \ \vee \ x\not= 0 \ $ and the valid dynamical rule $\ x=0,\
x\not= 0
\,\vdash\, \bot$.
A fortiori the same result is true for the theory of algebraically closed
fields.

For quasi-domains, the simplification axioms imply the same result.

Let us prove first that $p= 0$ has a dynamical proof from ${\cal
K}=(G;\Rzero,\Rnz )$
in the theory of quasi-domains if and only if ${\cal K} \cup ( p\not= 0 )$
collapses.
The ``only if" part follows from the valid dynamical rule $\ (x=0,\ x\not=0)
\,\vdash\, \bot$.
Suppose now that we have $m p^n +i=0$ with $m \in \Mnz$ and $i\in \Izero$ where
$\Mnz$ is the monoid generated by $\Rnz$ and $\Izero$ is the ideal
generated by
$\Rzero$. The presentation ${\cal K} $ proves $-i=0$ since $-i\in \Izero$,
so it proves $mp^n=0$ since $mp^n=(mp^n +i)+(-i)$. Then we deduce
$p^n =0$ using axiom $S(3)_f$ hence $p= 0$ using several times axiom $S(2)_f$.

Finally let us prove that $p\not= 0$ has a dynamical proof from ${\cal K}$
in the theory of quasi-domains if and only if ${\cal K} \cup ( p= 0 )$
collapses
in
the theory of rings with proper monoid. Suppose that we have $m+i+pa=0$
with $m
\in {\cal M}$ and $i\in \Izero$ We deduce $pa\not= 0$ and then $p\not= 0$
using $S(4)_f$.

So the theories of quasi-domains, fields and algebraically closed fields prove
the same facts since they collapse simultaneously.
\end{proof}

\begin{remark} \label{remdomains}
{\rm If we take the theory of rings with proper monoid and add the axiom $$
\,\vdash\, x = 0\ \vee\ x \not= 0 $$
we get the theory of {\em domains}. It is easy to see that axioms of
quasi-domains
 are valid dynamical rules for domains. Moreover it is interesting to
remark that the algorithm in Proposition \ref{prop-dyna-elim} gives
a collapse in the theory of domains. The theories of quasi-domains, domains,
fields and algebraically closed fields prove the same facts. Moreover they
collapse simultaneously with the theory of rings with
proper monoid.

This has interesting consequences for Heyting fields (see \cite{MRR88}).
Heyting
fields are a weak notion of field:
the equality relation $x=0$ is equivalent to $\lnot(x\not= 0)$ (where we
interpret $x\not= 0$ as meaning
the invertibility of $x$), but
the law of the excluded middle $x=0 \lor x\not= 0$ is not assumed, Heyting
fields satisfy axioms of quasi-domains and axioms of local rings
$$x\not= 0 \,\vdash\, \exists y \ yx-1=0 \qquad {\rm and}\qquad \,\vdash\,
x\not= 0\ \vee \ 1+x\not= 0\ $$ but it seems that there is no purely dynamical
description of an axiomatic for Heyting fields.
This is because there are no dynamical axioms for saying that $\lnot P$
means $P
\,\vdash\, \bot\;$.

A consequence of the Nullstellensatz is that any fact within a Heyting field
which can be
proved
in the theory of algebraically closed
fields can also be
proved
in the theory of Heyting fields.
So, when dealing with facts in a Heyting field, we can use freely all the
axioms
of algebraically closed fields. In particular the axiom $Dy(2)_f$ meaning the
decidability of equality to 0 causes no trouble with facts.}
\end{remark}

\section{Stengle's Positivstellensatz} \label{sec-
Stengle-Pos}

We give in this paragraph a new constructive proof of Stengle's
Positivstellensatz \cite{Ste74}.
This new proof is close to \cite{Lom89}.

\subsection{Some simultaneous collapses} \label{subsec-Steng-simcol}

The central theory we consider is the theory of ordered fields. The {\em unary
language of ordered fields} $\cl_{of}$ is the unary language of rings $\cl_r$
with two more unary predicates $\ge 0$ and $>0$.

Axioms of {\em proto-ordered ring} are axioms of rings
and the following axioms.
$$\begin{array}{rlccc}
&\,\vdash\, x^2 \ge 0&\qquad&\qquad& D(1)_{of} \\ 
x = 0 ,\ y \ge 0 &\,\vdash\,x+y \ge 0 &\qquad&\qquad& D(2)_{of} \\ 
x \ge 0 ,\ y \ge 0 &\,\vdash\, x+y \ge0&\qquad&\qquad& D(3)_{of} \\ 
x \ge 0,\ y \ge 0 &\,\vdash\, xy \ge 0&\qquad&\qquad& D(4)_{of} \\ 
& \,\vdash\, 1 > 0&\qquad&\qquad& D(5)_{of} \\
x=0,\ y > 0& \,\vdash\, x+y > 0 &\qquad&\qquad&D(6)_{of}\\
x > 0,\ y \ge 0& \,\vdash\, x+ y > 0&\qquad&\qquad& D(7)_{of} \\ 
x > 0,\ y> 0&\,\vdash\, xy > 0&\qquad&\qquad& D(8)_{of} \\ 
x > 0 &\,\vdash\, x \ge 0&\qquad&\qquad&D(9)_{of} \\ 
0 >0&\,\vdash\, \bot&\qquad&\qquad& C_{of} \\
\end{array}$$

Axioms of {\em ordered fields} are axioms
of proto-ordered rings and the following axioms. 
$$\begin{array}{rlccc}
x \ge 0,\ -x \ge 0 & \,\vdash\, x= 0 &\qquad&\qquad&S(1)_{of}\\ 
xy-1=0&\,\vdash\, x^2>0&\qquad&\qquad&S(2)_{of} \\ 
x\geq0 ,\ x^2 > 0 \ &\,\vdash\, \ x > 0&\qquad&\qquad&S(3)_{of} \\ 
x^2 > 0 &\,\vdash\, \exists y \ xy-1=0&\qquad&\qquad&Dy(1)_{of} \\ 
&\,\vdash\, x \ge 0\ \vee \ -x \ge 0&\qquad&\qquad&Dy(2)_{of} \\ 
&\,\vdash\, x = 0 \ \vee \ x^2 >0&\qquad&\qquad&Dy(3)_{of} \\ 
\end{array}
$$

Remark that if we introduce $x\neq 0$ as an abbreviation for $x^2>0$ then the axioms of rings with proper monoid are valid dynamical rules in the theory of proto-ordered rings.

The set of elements $x$ of a proto-ordered ring satisfying $x\geq0$ is a {\em proper cone}. Recall that a subset $C$ of a ring $A$ is
called a {\em cone} if the squares of $A$ are in $C$, $C+C\subset C$ and $C\,C \subset C$. A cone $C$ is said to be {\em proper} if $-1\not\in C$.

We write $t\ge t'$ as an abbreviation for $t-t'\ge 0$, $t\ge t' \ge t"$ as an abbreviation for $t\ge t',\ t'\ge t"$ and $t'\le t$ as another way of writing $t\ge t'$.

A {\em real closed field} is an ordered field with the extra axioms: $$ -
p(a)p(b) \geq 0 \,\vdash\, \exists y\ \ p(y)= 0\ \ \ \ Dy_n(4)_{of}$$ ($a,\ b,\ y$ and coefficients of the monic degree $n$ polynomial $p$ are distinct variables, and there is an axiom for each degree). 
Of course, 
$$ -p(a)p(b) \geq 0,\ b-a\geq 0
\,\vdash\, \exists y\ (p(y)= 0,\ b-y\geq0,\ y-a\geq0) 
$$
is a valid dynamical rule in the theory of real closed fields.\par

Notice that there are again direct algebraic axioms, collapse axioms,
simplification axioms and dynamical axioms.

A presentation in the language $\cl_{of}$ is a set of variables $G$ and three subsets $\Rzero, \Rnng, \Rpos$ contained in $\zg $. 
It is denoted by $(G;\Rzero,\Rnng,\Rpos)$.

The theory of proto-ordered rings has been chosen because the collapse
takes a very simple form and is very easy to prove as we shall see immediately (this is due to the fact that it is a direct theory), and because it collapses simultaneously with the theory of real closed fields as we shall see next.

\begin{proposition} \label{Lemme 3.3 }
Let ${\cal K}=(G;\Rzero,\Rnng,\Rpos)$ be a presentation in the language
$\cl_{of}$.
Let $\Izero$ be the ideal of $\zg $ generated by $\Rzero$, $\Mpos$
the monoid generated by $\Rpos$ ($\Mpos$ contains at least the element 1), $\Cnng$ the cone generated by $\Rnng\ \cup \ \Rpos$. A collapse of the presentation ${\cal K}$ in the theory of proto-ordered rings produces an equality in $\zg\;$:
$$m+ q + i \= 0
$$
with $m\in \Mpos$, $q\in \Cnng$ and $i\in \Izero$. Reciprocally, such an
equality produces a collapse of ${\cal K}$. \end{proposition}

\begin{proof}{Proof:}
First consider dynamical proofs of facts using {\em only direct algebraic axioms}. 
These are algebraic proofs without branching.

Arguing inductively on the number of times the direct algebraic axioms are
used in the proof we see successively that:

$\bullet$ provably $= 0$ elements, are exactly elements of $\Izero$,

$\bullet$ provably $\ge 0$ elements, are exactly elements of the form $q+i$ with $q \in \Cnng$ and $i\in \Izero$,

$\bullet$ provably $> 0$ elements, are exactly elements of the form $m+q+i$

with $m \in \Mpos$, $q\in \Cnng$ and $i\in \Izero$.

Now a proof of collapse is given by a proof of $0> 0$ using only direct
algebraic axioms. Necessarily it produces an equality $m+q+i\= 0$ in $\zg$.
\end{proof}

\begin{remark}\label{inclusion}
{\rm In the line of Remark \ref{ordre-des-predicats}, we can see in the
preceding
proof an order between the unary predicates we have in the language: first
comes
$=0$, then
$\geq 0$, and last $> 0$, and the collapse is concerned with this last
predicate $>0$. The axioms $D(1)_{of}$ to $D(4)_{of}$ are construction axioms
for $\geq 0$, and
the axioms $D(5)_{of}$ to $D(8)_{of}$ are construction axioms for $>0$.
Actually, the
situation here is a little more complex, since the remaining direct algebraic
axiom $D(9)_{of}$ contains $>0$ at the left and $\geq 0$ at the right, which
violates the order of construction. This axiom plays a special role: it
expresses the inclusion of $>0$ in $\geq 0$. One can realize that any proof in
the theory
of proto-ordered rings can be transformed into a proof where the axiom
$D(9)_{of}$ is only used at the beginning, i.e., before the application of any
other axiom. Indeed, any axiom $D(i+4)_{of}$ of construction for $>0$ is
doubled
by an axiom $D(i)_{of}$, and it is easy to check that an application of
$D(9)_{of}$ following an application of $D(i)_{of}$ can be transformed to an
application of $D(9)_{of}$ preceding an application of $D(i+4)_{of}$. To make
things
clearer, let us take an example. If $s=0$ and $t>0$, then $s+t>0$ by
$D(6)_{of}$
and
$s+t\geq0$ by $D(9)_{of}$; but one could also use first $D(9)_{of}$ to have
$t\geq0$,
then $D(2)_{of}$ to have $s+t\geq0$.\par

Another way of formulating this remark is to say that the axiom $D(9)_{of}$
may
be replaced by the stipulation that any presentation must satisfy
$\Rpos\subset
\Rnng$. This reflects the fact that the
construction
of the cone $\Cnng$
starts with $\Rnng\cup\Rpos$. So we have to distinguish between the direct
algebraic axioms expressing construction of
predicates, and those expressing an inclusion of predicates (here
$D(9)_{of}$).
The axioms of inclusion violate the order of construction, but they can be
lifted at the beginning of proofs. This preserves
the possibility of constructing the predicates one after the other.
}\end{remark}

Proposition \ref{Lemme 3.3 } means that a collapse of a presentation
in the
theory of proto-ordered rings can always be certified by an algebraic identity
of a very precise type.

Let us return to the example of the presentation $\ (x^3-y^3 = 0,\ (x-y)^2
>0,\
x= 0)\ $ which collapses in the theory of proto-ordered rings. This is
certified
by the equality $$(x-y)^4+y(x^3-y^3)-(x^3-3x^2y+6xy^2-4y^3)x=0 \eqno{(Ex_1)}$$
since $(x-y)^4$ belongs to the monoid generated by $(x-y)^2$ and $y(x^3- y^3)-
(x^3-3x^2y+6xy^2-4y^3)x$ belongs to the ideal generated by $x^3-y^3$ and $x$.

Similarly, the presentation
$\ ( x^3-y^3 = 0,\ (x-y)^2 >0,\ x ^2 >0)\ $ collapses in the theory of
proto-ordered
rings and this is certified by the equality
$$(x-y)^2x^2+2(x-y)^2x^2+(x-
y)^2(2y+x)^2-4(x-y)(x^3-y^3)=0 \eqno{(Ex_2)}$$ since
\begin{itemize}
\item $(x-y)^2x^2$ belongs to the monoid generated by $(x-y)^2$ and $x^2$,
\item
$2(x-y)^2x^2+(x-y)^2(2y+x)^2$ belongs to the cone generated by $(x-y)^2$,
\item
$4(x-y)(x^3-y^3)$ belongs to the ideal generated by
$(x^3- y^3)$.
\end{itemize}

\begin{theorem} \label{Corollaire 3.5}
The theory of ordered fields and the theory of proto-ordered rings collapse
simultaneously.
\end{theorem}

\begin{proof}{Proof:}
We are going to prove the following lemma, with abuses of notations
similar to
Notation \ref{nota-abus}.
\begin{lemma} \label{Lemme 3.4}
Let ${\cal K}=(G;\Rzero,\Rnng,\Rpos)$ be a presentation in the language
$\cl_{of}$, $p,r\in \zg $ and $z$
a new variable.

a) If the presentation
${\cal K}\cup ( p= 0 )$
collapses in the theory of proto-ordered rings, then so does the presentation
${\cal K}\cup ( p\ge 0, \ -p\ge 0 )$.

b) If the presentation
${\cal K} \cup (p^2>0 )$
collapses in the theory of proto-ordered rings, then so does the presentation
${\cal K} \cup (pr-1=0 )$.

c) If the presentation ${\cal K}\cup ( p> 0 )$
collapses in the theory of proto-ordered rings, then so does the presentation
${\cal K}\cup ( p\ge 0, \ p^2> 0 )$.

d) If the presentation
${\cal K} \cup (pz-1=0 )$
collapses in the theory of proto-ordered rings, then so does the presentation
${\cal K} \cup (p^2>0 )$.

e) If the presentations
${\cal K}\cup ( p\ge 0 )$ and
${\cal K}\cup ( -p\ge 0 )$ collapse in the theory of rings with proper cone,
then so does the presentation ${\cal K}$.

f) If the presentations
${\cal K}\cup ( p=0 )$ and
${\cal K}\cup ( p^2>0 )$
collapse in the theory of proto-ordered rings, then so does the presentation
${\cal K}$.
\end{lemma}

Lemma \ref{Lemme 3.4}
proves that the five additional axioms of ordered fields do not change the
collapse, which proves the theorem.

Let us prove now the lemma.

Let $\Mpos$ be the monoid generated by $\Rpos$ in $\zg $, $\Cnng$
the cone
generated by $\Rpos\cup \Rnng$ and $\Izero$
the ideal generated by
$\Rzero$.

a) We start with one identity $\ m+q+i=pb\ $ in $\zg $ with $m\in \Mpos$,
$q\in
\Cnng$, $i\in \Izero$ and $b\in \zg $. Squaring, we get an identity
$m_1+q_1+i_1\= p^2b^2\ $ and we rewrite it as $m_1+q_1+(p)(-p)b^2+i_1 \= 0\ $
which gives the collapse we are looking for.

b) Left to the reader (see the analogous computation in Lemma \ref{Lemme 2.4}.

c) We start with one identity $\ mp^{n}+q_1+q_2p+i=0\ $ in $\zg $ with $m\in \Mpos$, $q_1,q_2\in \Cnng$, $i\in \Izero$. If $n$ is even, we are done since this identity is also a collapse for ${\cal K}\cup ( p\ge 0, \ p^2> 0 )$. If $n$ is odd, then $\ mp^{n+1}+q_1p+q_2p^2+ip=0\ $ is a collapse for ${\cal K}\cup ( p\ge 0, \ p^2> 0 )$.

d) This is again Rabinovitch's trick.
We can assume that $p$ is not $0\in \zg$. There is an equality in $\Z[G,z]\;$:
$$ m + \sum _j q_jb_j(z)^2 + i(z) + (pz - 1) b(z) \= 0$$ with $m$ in $\Mpos$,
$q_j\in \Cnng$, $b_j$ and $b$ in $\Z[G,z]$, and $i(z)$ is a polynomial with
coefficients in $\Izero$.

Multiply by $p^{2n}$ where $2n$ is bigger than the $z$ degree of the
polynomials
$i(z)$ and $b_j(z )^2$. Replace in $(p^nb_j(z ))^2$ and in $p^{2n} i(z )$ all
$p^kz^k$ by 1 modulo $(pz - 1)$. The new polynomial $b(z)$ is
necessarily $0$ and since there is no more $z$ in what remains, we get an
equality which gives the collapse we are looking for.

e) We start with two identities in $\zg$ $$m_1+q_1+q'_1p+i_1\= 0 \ {\rm
(1)}\qquad {\rm and }\qquad m_2+q_2-q'_2p+i_2\= 0\ {\rm (2)} $$
with $m_1$ and $m_2\in \Mpos$, $q_1,q'_1$ and $q_2,q'_2\in \Cnng$ and $i_1$
and
$i_2$ in $\Izero$.
From (1) we deduce
$-q'_1p\= m_1+q_1+i_1$ and from (2) $q'_2p\= m_2+q_2+i_2$. Multiplying
these two
equalities we get
$- q'_1q'_2p^2\= (m_1+q_1+i_1)(m_2+q_2+i_2)$ and since $q'_1q'_2p^2$ is in
$\Cnng$,
this can be rewritten $m+q+i\= 0$ in $\zg$ with $m\in \Mpos$, $q\in \Cnng$,
$i\in \Izero$.

f) We start with two identities in $\zg$ $$p^{2n}m_1+q_1+i_1\= 0 \ {\rm (1)}
\qquad {\rm and} \qquad m_2+q_2+a p+i_2\= 0\ {\rm (2)}$$
with $m_1$ and $m_2\in \Mpos$, $q_1$ and $q_2$ in $\Cnng$ and $i_1$ and
$i_2$ in
$\Izero$.
Using (2), we get
$\ a^{2n}p^{2n}\= (m_2+q_2+i_2)^{2n}\= m_3+q_3+i_3\ $ (3) with $m_3$ in
$\Mpos$,
$q_3\in \Cnng$ and $i_3\in \Izero$. Multiplying now (1) by $a^{2n}$ and
substituting $p^{2n}a^{2n}$ by $m_3+q_3+i_3$ using (3), we obtain an equality
$m_4+q_4+i_4\= 0$ in $\zg$.
This gives the collapse we are looking for. \end{proof}

The proof of the lemma gives very explicit methods for constructing identities
certifying collapses. For instance, in our example, from the algebraic
identities $(Ex_1)$ and $(Ex_2)$ certifying that the presentations
$(x^3-y^3 = 0, (x-y)^2 >0, x= 0)$ and $(x^3-y^3 = 0, (x-y)^2 >0, x ^2 >0)$
collapse in the theory of proto-ordered rings, we can deduce an algebraic
identity certifying that the presentation $(x^3-y^3 = 0, (x-y)^2 >0)$
collapses in the theory of proto-ordered rings as in the preceding lemma e):
since $(x^3-3x^2y+6xy^2-4y^3)x\= (x-y)^4+y(x^3-y^3)$, $(x^3-3x^2y+6xy^2-
4y^3)^2x^2\= ((x-y)^4+y(x^3-y^3))^2$ and using $(Ex_2)$ multiplied by
$(x^3-(x^2y+6xy^2-4y^3)^2$ and replacing $(x^3-3x^2y+6xy^2-4y^3)^2x^2$ by $((x-
y)^4+y(x^3- y^3))^2$ we get an expression
$(x-y)^6+{\rm a\ sum\ of\ squares\ } +(x^3-y^3)A(x,y)\= 0$ which is the
algebraic identity we are looking for.
\begin{corollary} \label{Corollaire 3.7} Let $K$ be a real field
(i.e., $-1$ is not a sum of squares in $K$). The diagram of $K$ in the
language $\cl_{of}$ does not collapse in the theory of ordered fields
\end{corollary}

\begin{proof}{Proof:} Apply Theorem \ref{Corollaire 3.5} with the presentation
$\dg(K) \ \cup \ (\emptyset ; \emptyset , C,\emptyset )\;$, where $C$ is the
subset of sums of squares.
\end{proof}

This corollary is a constructive version of the non-constructive theorem
according to which ``every real field can be totally ordered".
Next theorem gives a constructive version of the fact that ``every ordered
field can be embedded in a real closed field".

\begin{theorem} \label{Proposition 3.8}
The theory of real closed fields and
the theory of ordered fields collapse simultaneously. \end{theorem}
\begin{proof}{Proof:}
We prove that the use of the extra axiom of real closed fields
$$-p(a)\,p(b)\ge 0 \,\vdash\, \exists z \ p(z) = 0$$ does not
modify the collapse. This is the content of the following lemma. Due to the
induction procedure, we have to consider polynomials $p(z)$ which may be
non-monic.
\end{proof}

\begin{lemma}
\label{lem-art-schr}
Let ${\cal K}=(G;\Rzero,\Rnng,\Rpos)$ be a presentation in the language
$\cl_{of}$, $a$ and $b$
elements of $\zg$, $z$
a new variable and
$p(z)\in \zg[z]$ (not necessarily monic). If the presentation ${\cal K}\cup (
p(z) = 0 )$
collapses in the theory of ordered fields, then so does the presentation
${\cal
K}\cup ( -p(a)p(b) \ge 0 )$
\end{lemma}

\begin{proof}{Proof}
Let $\Mpos$ be the monoid generated by $\Rpos$, $\Cnng$ the cone generated by
$\Rnng \cup \Rpos$ and $\Izero$ the ideal generated by $\Rzero$.

The proof is by induction on the formal degree of $p$ in $z$ (we say ``formal''
because $p(z)$ is not necessarily monic). For degree $0$ and $1$ it is easy.
Suppose now that $\deg(p) \geq 2$.

Consider first the case when $p$ is monic. From the collapse of ${\cal K}\cup (
p(z) = 0 )$, we obtain
an equality of polynomials in the variable $z\;$: $$m +\sum _i p_i s_i^2(z)
+\sum _j n_j t_j(z) + p(z) t(z) \= 0\ \eqno{(1)}$$ with $m\in \Mpos$, the
$p_i$
in
$\Cnng$ and the $n_j$ in $\Izero$. This is an algebraic identity in $\Z[G,z]$.
We
divide
$s_i$ and $t_j$ by $p$ and obtain an equality: $$m +\sum _i p_i r_i^2(z) +\sum
_j n_j q_j(z) - p(z) q(z) \= 0 \eqno{(2)}$$
with $m\in \Mpos$, the $p_i$ in $\Cnng$, the $n_j$ in $\Izero$ and ${\rm
deg}(q(z))\leq p-2$.

This equality provides a collapse of the presentation ${\cal K} \cup
(q(z)=0)$
in the theory of ordered fields. By induction hypothesis, we have thus a
collapse of the presentation ${\cal K} \cup (-q(a)q(b) \ge 0)$

On the other hand,
substituting $a$ (resp. $b$) to $z$ in equality (2), we obtain:
$$\displaylines{
\hfill m+ \sum _i p_i r_i^2(a)+ \sum _j n_j q_j(a) \= p(a) q(a)\hfill\llap{(3)}
\cr \hfill m+ \sum _i p_i r_i^2(b)+ \sum _j n_j q_j(b) \= p(b) q(b)\hfill
\llap{(4)}}$$

Equalities (3) and (4) show that in the presentation ${\cal K}$ the atomic
formulas
$p(a) q(a) > 0 $ and $p(b) q(b) > 0 $ are provable, hence also $p(a) p(b) q(a)
q(b)) > 0 $. \\ Thus the presentation ${\cal K} \cup (-p(a)p(b) \ge 0)$
proves
that $-q(a)q(b) \ge 0$
(it is easy to see that the dynamical rule $(xy>0,\ x\ge 0)\ \,\vdash\, \ y\ge
0$ is valid in the theory or ordered fields) and collapses.

In the case that $p$ is not monic, it is possible to open two branches. In
the first one, the leading coefficient of $p$ is zero, and the induction
hypothesis can be used. In the second branch, the leading coefficient of $p$
is invertible and we are reduced to the monic case using the axiom
$Dy(1)_{of}$
of ordered fields.
\end{proof}

\begin{theorem} \label{coroproposition 3.8} The theory of real closed
fields and
the theory of proto-ordered rings collapse simultaneously. \end{theorem}

\begin{corollary} \label{Proposition 3.8bis} Let ${\cal
K}=(G;\Rzero,\Rnng,\Rpos)$
be a presentation in the
language $\cl_{of}$. Let $\Izero$ be the ideal of $\zg$ generated by $\Rzero$,
$\Mpos$
the monoid generated by $\Rpos$,
$\Cnng$
the cone generated by $\Rnng \cup \Rpos$. A collapse of the presentation
${\cal
K}$ in the theory of real closed fields produces an equality in $\zg\;$:
$$m + q + i \= 0$$
with $m\in \Mpos$, $q\in \Cnng$ and $i\in \Izero$. \end{corollary}

\begin{proposition}
[non-constructive formal version of Stengle's Positivstellensatz]
\label{Proposition 3.9}
Let $A$ be a ring,
$(\Rzero,\Rnng,\Rpos)$
three families of elements. Denote by
$\Mpos$ the monoid generated by $\Rpos$, $\Cnng$ the cone generated by $\Rnng
\cup
\Rpos$, $\Izero$ the ideal generated by $\Rzero$. The following properties are
equivalent

i) There exists $i \in \Izero$, $p \in \Cnng$ and $m \in \Mpos$ with $m +
p+ i =
0$ in
$A$

ii) There exists no homomorphism $\phi : A \rightarrow L$ with $L$ real
closed,
$\phi (a)=0$ for $a \in \Rzero$, $\phi (p) \geq 0$ for $p\in \Rnng$ and $\phi
(m) > 0$ for $m\in \Rpos$.

\end{proposition}
\begin{proof}{Proof:}
Apply the preceding corollary to the presentation $$\dg(A) \cup (\emptyset;
\Rzero,\Rnng,\Rpos)\;,$$ and use the non-constructive completeness theorem of
model theory. \end{proof}

\subsection{Decision algorithm and constructive Positiv\-stellensatz} 
\label{subsec-eqtf-Steng-Pst}

In the next theorem, we mention the real closure of an ordered field. A
constructive proof of the existence and uniqueness (up to unique isomorphism)
of this real closure is for example given in \cite{LR90}. So the situation is
easier to describe than for algebraically closed fields and we can use more
directly semantics.

Since the theory of real closed fields has a decision algorithm for testing
emptiness with a very simple structure, we are able to prove
the following:

\begin{theorem}\label{th-
Stengle}
Let $K$ be an ordered field, $R$
its real closure, and $\Rzero$, $\Rnng$, $\Rpos$
three finite families of
$K[x_1,x_2,\cdots,x_n]=K[x]$. The system of sign conditions $[ u(x) >0, q(x)
\geq 0, j(x) =0]$ for $u \in \Rpos$, $q \in \Rnng$, $j \in \Rzero$ is
impossible
in $R^n$ if and only if the presentation
$$\dg(K) \ \cup (\{x_1,x_2,\cdots,x_n\};\Rzero,\Rnng, \Rpos)$$
collapses in the theory of real closed fields. \end{theorem}
\begin{proof}{Proof:}
We assume that from a constructive point of view, all our ordered fields are
discrete, this means that we have a way of deciding exactly if an element is
zero or not. Precisely, $G_1$ being the finite set of coefficients of
polynomials belonging to $\Rzero,\ \Rnng$, and $\Rpos$, we can decide for any
$\Z$-polynomial whether, when it is evaluated on $G_1$ in $K$, we get 0,
$>0$ or
$<0$.

We first deal with only one variable $x$. Recall Cohen-H\"ormander
algorithm. We
call H\"ormander tableau of a finite list of
polynomials with coefficients in the ordered field $K$ the tableau
\begin{itemize}
\item whose columns
correspond to roots of the polynomials in the real closure of $K$ and to open
intervals cut out by these roots, listed in the canonical order, \item and
which
has a line for each polynomial, whose entries are the sign ($>0$, $=0$ or
$<0$)
which the polynomial has at each of the roots or on each of the intervals.
\end{itemize}

\begin{lemma} \label{lem-tabl-Horm}
Let $K$ be an ordered field, subfield of a real closed field $R$. Let $L=[P_1,
P_2, \dots, P_k]$ be a list of polynomials of $K[x]$. Let $\cal P$ be the
family
of polynomials generated by the elements of $L$ and by the operations
$P\mapsto P'$, and $(P,Q)\mapsto {\rm Rem}(P,Q)$. Then:

$1)$ $\cal P$ is finite.

$2)$ One can set up the H\"ormander tableau for $\cal P$ using only the
following information:

$\bullet$ the degree of each polynomial in the family;

$\bullet$ the diagrams of the operations $P\mapsto P'$, and $(P,Q)\mapsto {\rm
Rem}(P,Q)$ (where $\deg(P)\ge deg(Q))$ in $\cal P$; and

$\bullet$ the signs of the constants of ${\cal P}$. \end{lemma}
\begin{proof}{Proof:}
1) is easy.

2) We number the polynomials in ${\cal P}$ so that the degree is
nondecreasing.
Let ${\cal P}_m$ be the subfamily of ${\cal P}$ made of polynomials numbered 1
to $m$. Let us denote by ${\cal T}_m $ the H\"ormander tableau
corresponding to
the family ${\cal P}_m\;$: i.e., the tableau where all the real roots of the
polynomials of ${\cal P}_m$ are listed in increasing order, and where all the
signs of the polynomials of ${\cal P}_m$ are indicated, at each root, and on
each interval between two consecutive roots (or between $-\infty$ and the
first
root, or between the last root and $+\infty$ ). Then by induction on $m$ it is
easy to prove that one can construct the tableau ${\cal T}_m$ from the allowed
information. \end{proof}
When one inspects the details of the preceding construction, one sees that it
means an elementary proof of a big disjunction (all the systems of sign
conditions for the list $L$ that appear when $x$ is in $R$). This elementary
proof leads directly to a covering of the presentation $\dg(K)\ \cup\
(\{x\};\emptyset,\emptyset,\emptyset)\ $ in the theory of real closed fields.

Now, if $\Rzero$, $\Rnng$, $\Rpos$ are three finite sets whose union is
$L$, and
if the corresponding sign conditions do not appear in the H\"ormander tableau,
we see that, considering the preceding covering as a covering of $\dg(K)\
\cup \
(\{x\}; \Rzero,\Rnng,\Rpos)$ we are able to ``kill" each leaf of the tree by a
collapse, since we get at each leaf a pair of contradictory sign
conditions on
at least one of the $P_i$'s.

\smallskip
Let us see now the multivariate case. We consider the variables $x_1$,\dots,\-
$x_{n-1}$ as parameters and the variable $x_n$ as our true variable. We try to
make the same computations as in the one variable case. Computations for
setting
the family ${\cal P}$ are essentially derivations and pseudo-remainder
computations. With coefficients depending on parameters, such a computation
splits in many cases, depending on the degrees of the polynomials (i.e.,
depending on the nullity or nonnullity of polynomials in the parameters).
Finally, the construction of
the H\"ormander tableau depends also on the signs of the ``constants" (i.e.,
some polynomials in the parameters) of the family ${\cal P}$. So, computing
all
possible signs conditions for a finite family of polynomials of
$K[x_1,\ldots,x_n]$ depends on computing all possible signs conditions for
another (much bigger) finite family of polynomials of $K[x_1,\ldots,x_{n-1}]$.

By induction we get a covering of the presentation $$\dg(K)\ \cup \
(\{x_1,x_2,\cdots,x_n\}; \emptyset,\emptyset,\emptyset)$$ in the theory of
real closed fields. At the leaves of our tree, we get all possible signs
conditions (when evaluated in $R^n$) for polynomials in $\Rzero\cup \Rnng\cup
\Rpos$. So if the system is impossible, the corresponding sign conditions
give a
collapse at each leaf of our tree when we consider this covering as a
covering
of the presentation $\dg(K)\ \cup \ (\{x_1,x_2,\cdots,x_n\};\Rzero,\Rnng,
\Rpos)$

Remark finally that a contrario, if one leaf of the tree has good sign
conditions, then we are able to explicit a point in $R^n$ satisfying the sign
conditions, and the presentation cannot collapse in the theory of real closed
fields.
\end{proof}

\begin{remark} \label{rem-GTF}
{\rm When we say that the correctness of the H\"ormander tableau has a very
elementary proof, we mean that the proof is only made of direct application of
our real closed fields axioms. The detailed inspection shows that only one
argument is ``indirect": the fact that a polynomial whose derivative is
positive
on an interval must be increasing on the interval. This fact has a very simple
proof based on algebraic identities (see e.g. in \cite{LR90} the ``algebraic
mean value theorem"). In the context of H\"ormander tableaux, these identities
lead to ``generalized Taylor formulas" (see \cite{Lom92}). Using these
formulas,
one gets a direct way of constructing the covering from the H\"ormander
tableau.
We can summarize this remark as ``H\"ormander tableaux and generalized Taylor
formulas produce dynamical proofs of collapses''. } \end{remark}

\begin{theorem}[Positivstellensatz] \label{Theoreme 3.10 } Let $K$ be an
ordered
field, $R$ its real closure, and $\Rzero,\Rnng,\Rpos$
three finite families of $K[x]=K[x_1,x_2,\cdots,x_n]$.

Define $\Mpos$ as the monoid generated by $\Rpos$, $\Cnng$ as the cone of
$K[x]$
generated by
$\Rpos\ \cup\ \Rnng\ \cup\ K^{>0}$ and $\Izero$ as the ideal of $K[x]$
generated
by $\Rzero$.

If the system of sign conditions $[ u(x) >0,\ q(x) \geq 0,\ j(x) =0]$ for $u
\in
\Rpos$, $q \in \Rnng$,
$j \in \Rzero$ is impossible in $R^n$
then one can construct an algebraic identity $$m+ p + i \= 0$$
where $m\in \Mpos$, $p\in \Cnng$ and $i\in \Izero$ \end{theorem}

\begin{proof}{Proof:} Apply the preceding theorem and Corollary
\ref{Proposition 3.8bis}.
\end{proof}

So the constructive character of Stengle's Positivstellensatz comes in our
approach from two different ingredients:
\begin{itemize}
\item the fact that the decision algorithm for testing emptiness produces,
when
the set realizing the presentation is empty in the real closure, a
collapse of
the presentation,
\item the fact that a collapse in the theory of real closed fields gives rise
to a construction of algebraic identities certifying this collapse.
\end{itemize}

\subsection{Provable facts and generalized Positivstellens\"atze}
\label{subsec-
provfacts-Steng-Pst}

We give in the next theorem some classical variants of Stengle's theorem.
It is
an immediate consequence of Theorem \ref{Theoreme 3.10 }.
\begin{corollary}\label{3.15}
Let $K$ be an ordered field, $R$ its real closure, $\Rzero$, $\Rnng$, $\Rpos$
three finite families of $K[x]=K[x_1,x_2,\cdots,x_n]$ and $p$ another
polynomial.

Define $\Mpos$ as the monoid generated by $\Rpos$, $\Cnng$ as the cone of
$K[x]$
generated by
$\Rpos\ \cup\ \Rnng \cup K^{>0}$, and $\Izero$ as the ideal of $K[x]$
generated
by $\Rzero$.

Let ${\cal S}$ the semialgebraic set of
$x\in R^n$ such that $u(x) >0$ for $u \in \Rpos$, $v(x) \geq 0$ for $v \in
\Rnng$, $j(x) =0$ for $j \in \Rzero$

a) The polynomial $p$ is nonzero on ${\cal S}$ if and only if one can
construct
an algebraic identity
$\ p b \= m + q + i \ $ with
$m\in \Mpos$, $i\in \Izero$, $q\in \Cnng$ and $b\in K[X]$.

b) The polynomial $p$ is positive on ${\cal S}$ if and only if one can
construct
an algebraic identity
$\ pq'\= m + q + i \ $ with $m\in \Mpos$, $i\in \Izero$ and $q,q'\in \Cnng$.

c) The polynomial $p$ is zero on ${\cal S}$ if and only if one can
construct an
algebraic identity
$\ p^{2n}m + q + i \= 0 \ $ with $m\in \Mpos$, $i\in \Izero$ and $q\in \Cnng$.

d) The polynomial $p$ is nonnegative on ${\cal S}$ if and only if one can
construct an algebraic identity
$\ p q \= p^{2n}m + q' + i\ $ with $m\in \Mpos$, $i\in \Izero$ and $q,q'\in
\Cnng$.
\end{corollary}

\begin{proof}{Proof:}
It is easy to see that in the theory of ordered fields a fact is provable
(from
a presentation) if and
only
if the ``opposite" fact
produces a collapse (when added to the presentation). This is because we have
the valid dynamical rules $\vdash x=0 \ \vee \ x^2> 0 \ $,
$\vdash x>0 \ \vee \ -x\ge 0 \ $,
$\ (x=0,\ x^2> 0) \,\vdash\, \bot \ $ and $\ (x>0,\ -x\ge 0) \,\vdash\, \bot$.
So
the corollary is an easy
consequence of Theorem \ref{Theoreme 3.10 }. \end{proof}

We give now the axioms of {\em quasi-ordered rings}: it is the theory of
proto-ordered rings together with the following simplification axioms:

$$\begin{array}{rlccc}
x^2 \leq 0 \ &\,\vdash\, \ x = 0&\qquad&\qquad&S(4)_{of} \\
x > 0 , \ x y \geq 0 \ &\,\vdash\, \ y \geq 0&\qquad&\qquad&S(5)_{of} \\ 
x\geq0 ,\ x y > 0 \ &\,\vdash\, \ y > 0&\qquad&\qquad&S(6)_{of} \\ 
c \geq 0 , \ x(x^2 + c)\ \geq 0 ) \ &\,\vdash\, \ x \geq 0 &\qquad&\qquad&S(7)_{of} \\
\end{array}$$

Remark that simplification axioms $S(1)_{of}$, $S(2)_{of}$ and $S(3)_{of}$ (given for ordered fields) are valid dynamical rules in the theory of quasi-ordered rings. Note also that the theory of quasi-ordered rings has only algebraic axioms and one collapse axiom.

An ordered field is a quasi-ordered ring. More precisely, axioms of
quasi-ordered
rings are axioms of ordered fields or valid dynamical rules in the
theory of ordered fields.
So quasi-ordered rings are between proto-ordered rings and ordered fields, and
we get the following lemma.

\begin{lemma} \label{lem-coll-sim-protoordring-ordfields} The theories of
proto-ordered rings, quasi-ordered rings, ordered fields and real closed fields
collapse simultaneously. \end{lemma}

\begin{proposition}
\label{3.17}
The theories of quasi-ordered rings, ordered fields and real closed fields
prove
the same facts.
\end{proposition}

\begin{proof}{Proof:}
We have already said that, in the theory of ordered fields, a fact is provable
(from a presentation) if and only if the opposite fact produces a collapse
(when added to the presentation). A fortiori the same result is true for the
theory of real closed fields.

For quasi-ordered rings, the simplification axioms give the same result (note
that the two collapse axioms are easy).

We give the more tricky case and leave the other ones to the reader.

Assume that the presentation
$(G;\Rzero,\Rnng,\Rpos\cup \{-p\})$ collapses in the theory of proto-ordered
rings. So we get an equality $(-p)^\ell m+ q + i \= p q'$ in $\zg$ with $m$ in
the monoid $\Mpos$ generated by $\Rpos$, $i$ in the
ideal $\Izero$ generated by $\Rzero$, and $q$ and $q'$ in the cone $\Cnng$
generated by $\Rpos\cup \Rnng$. We may assume that $\ell=2n$ is even (if
not, we
multiply by $-p$ and rewrite the equality). So we have $(p^n)^2 m + q \= p q'-
i$. We may assume that $n$ is odd (if not multiply by
$p^2$). Multiplying by $p^n$, we get an equality $p^n((p^n)^2 m + q) \= q_1 +
i_1$.
Hence, $p^n((p^n)^2 m+ q) \ge 0$.
Applying $S(7)_{of}$ we get $p^n \ge 0$ with $n$ odd. A consequence of
$S(7)_{of}$ is the simplification rule $x^3\ge 0 \,\vdash\, x\ge 0$, which
allows to deduce here $p\ge 0$ (multiply $p^n$ by an even power of $p$ in
order
to get $p^{3^k}\ge 0$).

So the theories of quasi-ordered rings, ordered fields and real closed fields
prove the same facts since they collapse simultaneously. \end{proof}

\section{A Positivstellensatz for valued fields} \label{sec-pst-valfields}

We give in this section a new ``Positivstellensatz" for valued
fields. As we obtained in the last paragraph a constructive analog of the
classical theorem ``every real field can be totally ordered" we shall obtain
here as a consequence a constructive
version of the following theorem
``the intersection of valuation rings of a field $K$ containing a subring
$A$ is
the integral closure of $A$" (Corollary \ref{cor-clotint}).

\subsection{Some simultaneous collapses} \label{subsec-valf-simcol}

We need to consider a subring $A$ of a field $K$. The language $\cl_v$ will
include the language $\cl_1$ with its two unary predicates $= 0$ and $\not=
0$,
and three more unary predicates $\Vr(x)$, $\Rn(x)$ and $\U(x)$ corresponding
respectively to the elements of the valuation ring, the elements becoming zero
in the residue field and the elements becoming invertible in the residue
field.

A presentation in the language $\cl_v$ is a set of variables $G$ and five
subsets $\Rzero, \Rnz, \Rvr,\Rrn,\Ru$ of $\zg$.
It is denoted by $(G;\Rzero,\Rnz, \Rvr,\Rrn,\Ru)$.

The most basic notion is the notion of valued field. The structure of
proto-valued
rings will be the simplest direct theory that we shall consider. We
introduce this theory because it is a direct theory which collapses
simultaneously with the theory of valued fields.

The axioms for {\em proto-valued rings} are axioms of rings and the following
axioms.
$$\begin{array}{rlccc}
&\,\vdash\, \Vr(-1) &\qquad&\qquad&D(1)_v \\
x = 0 ,\ \Vr(y) &\,\vdash\, \Vr(x+y)&\qquad&\qquad&D(2)_v\\ \Vr(x) ,\ \Vr(y)
&\,\vdash\, \Vr(xy) &\qquad&\qquad&D(3)_v\\ \Vr(x) ,\ \Vr(y) &\,\vdash\,
\Vr(x+y) &\qquad&\qquad&D(4)_v\\ &\,\vdash\, \Rn(0) &\qquad&\qquad& D(5)_v \\
x = 0 ,\ \Rn(y) &\,\vdash\, \Rn(x+y)&\qquad&\qquad& D(6)_v \\ \Rn(x) ,\
\Vr(y) &
\,\vdash\, \Rn(xy) &\qquad&\qquad&D(7)_v \\ \Rn(x) ,\ \Rn(y) &\,\vdash\,
\Rn(x+y) &\qquad&\qquad&D(8)_v\\ &\,\vdash\, \U(1) &\qquad&\qquad& D(9)_v\\
x = 0 ,\ \U(y) &\,\vdash\, \U(x+y) &\qquad&\qquad&D(10)_v \\ \U(x) ,\ \U(y)
&\,\vdash\, \U(xy)&\qquad&\qquad&D(11)_v\\
\Rn(x) ,\ \U(y) &\,\vdash\, \U(x+y)&\qquad&\qquad&D(12)_v \\ \U(x)
&\,\vdash\, x
\not= 0 &\qquad&\qquad& D(13)_v\\
x = 0 ,\ y \not= 0 &\,\vdash\, x+y \not= 0&\qquad&\qquad&D(14)_v\\ x \not=
0 ,\
y\not= 0 &\,\vdash\, x y \not= 0 &\qquad&\qquad&D(15)_v\\ \U(x) &\,\vdash\,
\Vr(x) &\qquad&\qquad&D(16)_v\\ \Rn(x) &\,\vdash\,
\Vr(x)&\qquad&\qquad&D(17)_v
\\ (0 \not= 0)&\,\vdash\, \bot &\qquad&\qquad& C_v \\ \end{array}$$

We add now the following axioms for {\em valued fields}
$$\begin{array}{rlccc}
x u -1=0 &\,\vdash\, x \not= 0 &\qquad&\qquad&S(1)_v\\ \Vr(xy),\ \U(x)
&\,\vdash\, \Vr(y) &\qquad&\qquad&S(2)_v\\ x \not= 0 &\,\vdash\, \exists u\
xu-1=0 &\qquad&\qquad&Dy(1)_v\\
&\,\vdash\, x = 0 \ \vee \ x \not= 0
&\qquad&\qquad&Dy(2)_v\\ xy =1 &\,\vdash\, \Vr(x) \ \vee \ \Vr(y)
&\qquad&\qquad& Dy(3)_v\\ \Vr(x) &\,\vdash\, \U(x) \ \vee \ \Rn(x)
&\qquad&\qquad&Dy(4)_v \\ \end{array}$$

Finally, the theory of {\em algebraically closed valued field } is obtained
when
adding the axioms of algebraic closure
$$\begin{array}{rlccc}
&\,\vdash\, \exists y \ y^n + x_{n-1} y^{n-1}+ \cdots+ x_1 y + x_0 = 0
&\qquad&\qquad&Dy_n(5)_v\\
\end{array}$$
\begin{remark}{\rm We can extend Remarks \ref{ordre-des-predicats}
and \ref{inclusion} to this new
theory. Here the order of the predicates is $=0$, $\Vr$, $\Rn$, $\U$, $\not=0$,
and the collapse concerns this last predicate. The inclusion axioms are
$D(16)_v$ and $D(17)_v$. Any proof using direct algebraic axioms may be
transformed to a proof where the inclusion axioms are used only at the
beginning.
The facts $\vdash \Vr(1)$ and $\vdash \Vr(0)$ can be proved from the
construction axioms for $\Vr$. The construction axioms for $\Rn$ and $\U$ are
doubled by construction axioms for $\Vr$ which allow to lift the inclusions at
the beginning.}
\end{remark}

The order on the predicates and the inclusion axioms that we have distinguished
agree with the characterization of the collapse of a presentation in the
theory
of
proto-valued rings, and its proof. This collapse is particularly simple.

\begin{proposition} \label{Lemme 4.3}
Let ${\cal K}=(G;\Rzero,\Rnz,\Rvr,\Rrn,\Ru)$ be a presentation in the language
$\cl_v$. Let $\Izero$ be the ideal of $\zg$ generated by $\Rzero$, $\Mnz$ the
monoid generated by $\Rnz$, $\Vvr$ the
subring generated by $\Rvr \cup \Rrn \cup \Ru$, $\Irn$ the ideal of $\Vvr$
generated
by $\Rrn$ and $\Mu$ the monoid generated by $\Ru$.

The presentation ${\cal K}$ collapses in the theory of proto-valued rings if
and
only if there is an equality in $\zg$ $$m (u+j) + i \= 0$$ with $m\in \Mnz$,
$u\in \Mu$, $j\in \Irn$ and $i\in \Izero$. \end{proposition}

\begin{proof}{Proof:}
First consider dynamical proofs of facts using {\em only direct algebraic
axioms}. These are algebraic proofs without branching.

Arguing inductively on the number of times the direct algebraic axioms are
used
in the proof we see successively that:

$\bullet$ provably $= 0$ elements, are exactly elements of $\Izero$,

$\bullet$ provably $\Vr$-elements, are exactly elements of the form $b+i$ with
$b\in \Vvr$ and $i\in \Izero$.

$\bullet$ provably $\Rn$-elements, are exactly elements of of the form $j+i$
with $j\in \Irn$ and $i\in \Izero$.

$\bullet$ provably $\U$-elements, are exactly elements of the form $u+j+i$
with
$u\in \Mu$, $j\in \Irn$ and $i\in \Izero$,

$\bullet$ provably $\not= 0$-elements, are exactly elements of the form
$m(u+j)+i$ with $m\in \Mnz$, $u\in \Mu$, $j\in \Irn$ and $i\in \Izero$.

Now a proof of collapse is given by a proof of $0\not= 0$ using only direct
algebraic axioms. Necessarily it produces an equality $m(u+j)+i\= 0$ in $\zg$.
\end{proof}

\begin{theorem}\label{th-coll-sim-prvr-valf-alclvalf} The theories of
proto-valued rings, valued fields and algebraically closed valued field
collapse
simultaneously.
\end{theorem}
\begin{proof}{Proof:}
The theorem is proved by induction on the number of times that the extra
axioms
for
algebraically closed valued fields are used. So it is enough to prove the
following lemma.
\begin{lemma}\label{Lemme 4.4}
Let ${\cal K}=(G;\Rzero,\Rnz,\Rvr,\Rrn,\Ru)$ be a presentation in the language
$\cl_v$, $p, q \in \zg$, $z$ a new variable and $r(z)$ a $z$-monic
polynomial in
$\zg[z]$.

a) If the presentation
${\cal K} \cup ( p\not= 0 )$ collapses in the theory of proto-valued rings,
then so
does the presentation ${\cal K}\cup (pq - 1 = 0 )$.

b) If the presentation
${\cal K} \cup (\Vr(q) )$
collapses in the theory of proto-valued rings, then so does the presentation
${\cal K} \cup (\Vr(pq), \U(p) )$.

c) If the presentation
${\cal K}\cup ( z p - 1 = 0 )$ collapses in the theory of proto-valued rings,
then so
does the presentation ${\cal K} \cup ( p\not= 0 )$.

d) If the presentations
${\cal K} \cup (p\not= 0 )$
and
${\cal K} \cup (p = 0 )$
collapse in the theory of proto-valued rings, then so does the presentation
${\cal K}$.

e) If the presentations
${\cal K} \cup ( \Vr(p))$
and
${\cal K} \cup ( \Vr(q))$
collapse in the theory of proto-valued rings, then so does the presentation
${\cal K} \cup ( qp - 1 = 0 )$

f) If the presentations
${\cal K} \cup (\U(p) )$
and
${\cal K} \cup (\Rn(p) )$
collapse in the theory of proto-valued rings, then so does the presentation
${\cal K} \cup
(\Vr(p) )$.

g) If the presentation
${\cal K} \cup (r(z)=0 )$
collapse in the theory of proto-valued rings, then so does the presentation
${\cal K}$.
\end{lemma}
\begin{proof}{Proof:}

We take the following notations. Letters $m, u, j, i, a, b$ (possibly with
indices) represent always respectively elements of $\Mnz, \Mu, \Irn, \Izero,
\Vvr,
\zg$
(defined as in Proposition \ref{Lemme 4.3}). The symbol $b(z)$ stands for a
polynomial with coefficients in $\zg$ and so on.

a) Left to the reader (see the analogous computation in Lemma \ref{Lemme 2.4}
a))

b) There is an equality $m ( u + j(q) ) + i \= 0$. If the polynomial $j$ is
of
degree $n$, multiplying the equality by $p^n$ gives $m( p^n u + j_1(p,pq) ) +
i_1 \= 0$
which is the collapse we want.

c) This is Rabinovitch's trick. There is an equality $$m (u+j(z)) + i(z) +
(z p
- 1) b(z) \= 0\;.$$ Multiply by $p^n$ where $n$ is the maximum $z$-degree
of the
polynomials $i(z)$ and $j(z)$.
Replace in $p^n i(z)$ and in $p^n j(z)$ all the $p^k z^k$ by $1$ modulo $(z
p -
1)$. We can assume that $p$ is not $0\in\zg$. The new polynomial
$b(z)$ is necessarily
$0$ and this gives the
equality $p^n m (u+j_1) + i_1 \= 0$
which is the collapse we want.

d) There are two equalities in $\zg\;$:
$$ 
p^n m_1 (u_1+j_1) + i_1 \= 0 \qquad {\rm and}\qquad m_2 (u_2+j_2) + i_2\= pb_2.\;
$$
Raise the second one to the power $n$, multiply the result by $m_1 (u_1+j_1)$,
multiply the first one by $b_2^n$ and combine the two equalities so
obtained in order to get the collapse we want.

e) There are two equalities in $\zg\;$:
$$m_1 (u_1 + j_1(p)) + i_1 \= 0 \qquad {\rm and}\qquad m_2 (u_2 + j_2(q)) +
i_2
\= 0\;.$$
Without
loss of generality we can suppose that $m_1 = m_2 = m$. If $n$ is the
degree in
$p$ of $j_1$ one multiplies the first equality by $q^n$. Modulo $(pq - 1)$ the
polynomial
$q^n (u_1+ j_1(p))$ can be rewritten as a {\em nice} polynomial in $q$,
$n_1(q)$, i. e. its leading coefficient is $u_1+j_{1,0}$ (in $\Mu+\Irn$)
and the
other coefficients are in $\Irn$. This gives an equality:
$$m n_1(q) + i_3 + (pq - 1) b_1 \= 0 \ \ \ \ \ (1)$$ Doing the same
manipulation
with the second equality gives $$m n_2(p) + i_4 + (pq - 1) b_2 \= 0 \ \ \ \ \
(2)$$ One can then compute two polynomials
$r_1(p,q)$ and $r_2(p,q) $ with coefficients in $\Vvr$ such that there is an
equality
$\ n_1(q) r_1(p,q) + n_2(p) r_2(p,q) \= n_3(pq)\ $ where $n_3$ is nice too:
take as $n_3$ the general polynomial whose roots are the products of a root of
$n_1$ and a root of $n_2$.

Multiplying (1) by $r_1(p,q)$ and (2) by $r_2(p,q)$ and adding, we obtain:
$$m n_3(pq) + i_5 + (pq - 1) b_5 \= 0\;.$$

It remains to replace the $pq$ in $n_3$ by 1 modulo $(pq - 1)$ to find the
wanted collapse: $\ m (u_6+j_6) + i_5 + (pq - 1) b_6 \= 0$.

f) There are two equalities
$$\ m_1 ( p^n u_1 + j_1(p) ) + i_1 \= 0
\qquad {\rm and}\qquad m_2 (u_2 + j_2 + p a_2(p) ) + i_2 \= 0\;.$$

Rewrite the second equality in the form
$ m_2 ( u_2 + j_2 ) \= - ( m_2 p a_2(p) + i_2 )$. Raise it to the power $n$ and
multiply by $m_1u_1$, so that the right-hand side becomes $(-
1)^nm_1m_2^np^nu_1a_2^p(n)+i_3$, and so on in order to get the collapse we
want
(last details to the
reader).

g) We have an equality
$\ m ( u + j(z) ) + i(z) \= r(z) b(z)\ $. Reduce $i$ and $j$ modulo $r$, the
right-hand side becomes identically zero and this gives an equality which is a
collapse of ${\cal K}$.
\end{proof}
\end{proof}

 \begin{corollary}
\label{cor-extension de valuation}
Let $(K,A)$ be a valued field and $L$ a field extension of $K$. Then the
presentation obtained from
diagrams of $(K,A)$ and $L$ does not collapse in the theory of valued fields.
\end{corollary}

\begin{proof}{Proof:} A collapse would give an equality $m(u+j)=0$ in $L$,
with $u$ invertible in $A$, $j$ in the maximal ideal of $A$ and $m$ nonzero in
$L$. But this implies $u+j=0$ in $L$, hence in $K$ and this is impossible.
\end{proof}

\begin{remark} \label{rem constructive version extension valued field} {\rm
The
preceding corollary is a constructive version of the non-constructive theorem
saying that a valuation of a field $K$ can always be extended to any field
extension $L$ of $K$. This non-constructive
theorem is a direct consequence of the corollary, obtained using completeness
theorem of model theory.
}
\end{remark}

In the same way we get a constructive version of the classical theorem saying
that a local subring of a field $K$ is always dominated by a valuation ring of $K$.

\begin{corollary}
\label{cor-local domine par val} 
Let $K$ be a field 
and $A\subset K$ a local ring. Then the presentation in the language $\cl_v$ obtained from the diagram of $K$ by adding  
$\U(a)$ when $a$ is invertible in $A$ and 
$\Rn(a)$ when $a$ is in the maximal ideal of $A$
does not collapse in the theory of valued fields. 
\end{corollary}

In the same way we get a ``formal Positivstellensatz for valued fields".

\begin{proposition}[formal non-constructive version of Positivstellensatz for
valued
fields]
\label{Proposition 4.9}
Let
$B$ be a ring and $(\Rzero,\Rnz,\Rvr,\Rrn, \Ru)$ subsets of $B$.
Let $\Izero$ be the ideal of $B$ generated by $\Rzero$, $\Mnz$ the monoid
of $B$
generated by $\Rnz$, $\Vvr$ the subring of $B$ generated by $ \Rvr \cup \Rrn
\cup \Ru$, $\Irn$ the ideal of $\Vvr$ generated by $\Rrn$, $\Mu$ the monoid
generated by $\Ru$.

The following properties are equivalent:

i) There exists $i \in \Izero$, $s \in \Mnz$, $u \in \Mu$ and $j \in \Irn$
with
$m ( u + j ) + i = 0$

ii) There exists no homomorphism $\phi : B \rightarrow L$ with $(L,A,I,U)$ an
algebraically closed valued field, $\phi (n) = 0$ for $n \in \Rzero$, $\phi
(t)
\not= 0$ for $t \in \Rnz$, $\phi (c) \in A$ for $c \in \Rvr$, $\phi (k) \in I$
for $k \in \Rrn$ and $\phi (v) \in U$ for $v \in \Ru$.
\end{proposition}

\begin{proof}{Proof:}
Use the preceding results taking as presentation $$\dg(B) \cup
(\emptyset;\Rzero,\Rnz,\Rvr, \Rrn,\Ru)\;,$$ and apply the non-constructive
completeness theorem of model theory. \end{proof}

\subsection{Decision algorithm and constructive Positiv\-stellensatz} \label{subsec-eqtf-valf-Pst}

\begin{theorem}[Positivstellensatz for algebraically closed valued fields]
\label{Theoreme 4.10}
Let $(K,A)$ be
a
valued field and $U_A$ the invertible elements of $A$, $I_A$ the maximal ideal
of $A$.
Suppose that $(K',A')$ is an algebraically closed valued field extension of
$K$
(so that $A=A' \cap K$). Denote by $U_{A'}$ the invertible elements of $A'$,
$I_{A'}$ the maximal ideal of $A'$.

Consider five finite families $(\Rzero,\Rnz,\Rvr,\Rrn,\Ru)$ in the polynomial
ring
$K[x_1,x_2,\cdots,x_m]=K[x]$.

Let $\Izero$ be the ideal of $K[x]$ generated by $\Rzero$, $\Mnz$ the monoid of
$K[x]$ generated by $\Rnz$, $\Vvr$ the subring of $K[x]$ generated by $\Rvr
\cup
\Rrn \cup \Ru \cup A$, $\Irn$ the ideal of $\Vvr$ generated by $\Rrn\cup I_A$,
$\Mu$ the monoid generated by $\Ru\cup U_A$.

Let ${\cal S} \subset {K'}^m$
be the set of points $x$ satisfying the conditions: $n(x) = 0$ for $n \in
\Rzero$,
$t(x) \not= 0 $ for $t \in \Rnz$,
$c(x) \in A' $ for $c \in \Rvr$,
$v(x) \in U_{A'}$ for $v \in \Ru$, $k(x) \in I_{A'}$ for $k \in \Rrn$.

The set ${\cal S}$ is empty if and only if there is an
algebraic identity
$$m (u+j) + i \= 0$$
with $m \in \Mnz$, $u \in \Mu$, $j \in \Irn$ and $i \in \Izero$. \end{theorem}

\begin{proof}{Proof:}
We have the following result (see e.g. \cite{Wei84} section 3 or \cite{klp00}
section 3).
The formal
theory ${\cal V}(K,A)$
of algebraically closed valued fields extensions of a valued field $(K,A)$ is
complete and has a decision algorithm using only computations inside $(K,A)$.

Consider now the presentation
$${\cal P}=\dg(K,A)\cup
(\{x_1,\ldots,x_m\};\Rzero,\Rnz,\Rvr,\Rrn,\Ru)$$ in
the language $\cl_v$. Since the system of sign conditions we consider is
impossible in $(K',A')$,
it is thus proved impossible in ${\cal V}(K,A)$. According to Theorem
\ref{Theoreme 1.3}, the presentation ${\cal P}$ collapses in the theory of algebraically
closed
valued fields. We conclude by Theorem \ref{th-coll-sim-prvr-valf-alclvalf} and Proposition \ref{Lemme 4.3}.
\end{proof}

Note that the same proof could have been used in the cases of algebraically
closed fields and real closed fields, but there we were able to prove directly
the existence of dynamical proofs because of particular features of the
decisions algorithms for testing emptiness we used in these two cases. In
fact,
the proof in \cite{klp00} can be transformed as well into an algorithm
producing
a dynamical proof in a more direct way.

\subsection{Provable facts and generalized Positivstellens\"atze} 
\label{subsec-provfacts-valf-Pst}

We now discuss provability of facts in the theory of valued fields.

We define a {\em quasi-valued ring} as a proto-valued ring satisfying the
following simplification axioms (the first one
is an axiom of valued fields).

$$\begin{array}{rlccc}
\Vr(xy),\ \U(x) &\,\vdash\, \Vr(y) &\ S(2)_v \\ \U(xy), \Vr(x), \Vr(y)
&\,\vdash\, \U(y) &S(3)_v \\
\Rn(xy),\ \U(x) &\,\vdash\, \Rn(y) &
S(4)_v \\ \Rn(x^2) &\,\vdash\, \Rn(x) & S(5)_v \\ xy \not= 0 &\,\vdash\, x
\not=
0 & S(6)_v \\
xy = 0,\ x \not= 0 & \,\vdash\, y
= 0 &S(7)_v \\ x^2 = 0 &\,\vdash\, x = 0 &S(8)_v \\
x^{n+1} -\sum_{k=0}^n a_k x^k=0 , \Vr(a_n), \cdots, \Vr(a_0) &\,\vdash\, \Vr(x)
&S(9)_v \\
\end{array}$$

The theory of quasi-valued rings is an algebraic theory and we shall see soon
that
it proves the same facts as the theory of algebraically closed valued fields.

It is easy to check the following lemmas.

\begin{lemma}
\label{lem-valf-valr}
Axioms of quasi-valued rings are valid dynamical rules in the theory of valued
fields.
\end{lemma}

\begin{proof}{Proof:} Let us prove for example that $S(9)_v$ is a valid
dynamical rule. Assume
$$x^{n+1} -\sum_{k=0}^n a_k x^k=0 ,\ \Vr(a_n), \cdots,\
\Vr(a_0)\;.\eqno{(1)}$$
Open two branches using axiom $Dy(2)_v$, the first one with $x=0$ (so
$\Vr(x)$)
and the second one with $x\not= 0$. Here use axiom $Dy(1)_v$ and introduce the
inverse $y$ of $x$ (so $xy=1$). Then use axiom
$Dy(3)_v$ and open two branches, the first one with $\Vr(x)$ (we are done) and
the second one with $\Vr(y)$. Multiply the equality in $(1)$ by $y^n$. Since
$xy=1$, we get $x =\sum_{k=0}^n a_k y^{n-k}$ and we deduce easily $\Vr(x)$.
\end{proof}

\begin{lemma} \label{lem-valr}
We have the following valid dynamical rules in the theory of quasi-valued
rings.
$$\begin{array}{rlccc}
\U(xy),\ \U(x) &\,\vdash\, \U(y) & S(10)_v \\
x^{n+1} -
\sum_{k=0}^n a_k x^k=0 , \Vr(a_n), \cdots, \Vr(a_1),\ \U(a_0) &\,\vdash\,
\U(x)
&S(11)_v \\ \end{array}$$
\end{lemma}

 \begin{lemma} \label{lem-
coll-sim-val}
The theories of proto-valued rings, quasi-valued rings, valued fields and
algebraically closed valued fields collapse simultaneously. \end{lemma}

 \begin{lemma}
\label{lem-valf-provfact-colaps} Let ${\cal K}=(G;\Rzero,\Rnz,\Rvr,\Rrn,\Ru)$
be a presentation in the language $\cl_v$. Let $p$ be an element of $\zg$ and
$z$ a new variable.

a) In the theory of valued fields, the fact $p = 0$ is provable from the
presentation ${\cal K}$ if and only if the presentation ${\cal K} \cup
(p\not= 0
)$ collapses.

b) In the theory of valued fields, the fact $p \not= 0$ is provable from the
presentation ${\cal K}$ if and only if the presentation ${\cal K} \cup (p=0 )$
collapses.

c) In the theory of valued fields, the fact $\Vr(p)$ is provable from the
presentation ${\cal K}$ if and only if the presentation ${\cal K} \cup ( zp-
1=0,\Rn(z) )$ collapses.

d) In the theory of valued fields, the fact $\Rn(p)$ is provable from the
presentation ${\cal K}$ if and only if the presentation ${\cal K} \cup ( zp-
1=0,\Vr(z) )$ collapses.

e) In the theory of valued fields, the fact $\U(p)$ is provable from the
presentation ${\cal K}$ if and only if the presentations ${\cal K} \cup (zp-
1=0,\Rn(z))$ and
${\cal K} \cup (\Rn(p))$ collapse.
\end{lemma}

Remark that the last lemma is a fortiori true for algebraically closed valued
fields. So theories of valued fields and algebraically closed valued fields
prove the same facts since they collapse simultaneously.

From the last lemma and the algebraic characterization of collapses of
presentations in the theory of valued fields, one can prove the following
proposition.

\begin{proposition}
\label{prop-factvalf} Let ${\cal K}=(G;\Rzero,\Rnz,\Rvr,\Rrn,\Ru)$ be a
presentation in the language $\cl_v$. Let $p$ be an element of $\zg$.
Define $\Izero$, $\Mnz$, $\Vvr$, $\Irn$ and $\Mu$ as in Proposition \ref{Lemme 4.3}.

a) In the theory of valued fields, a dynamical proof of the fact $p = 0$ from
the presentation ${\cal K}$
produces an equality in $\zg$ of the following type $$ p^n m(u+j) + i \= 0$$
with $m \in \Mnz$, $u \in \Mu$, $j \in \Irn$ and $i \in \Izero$.

b) In the theory of valued fields, a dynamical proof of the fact $p \not= 0$
from the presentation ${\cal K}$
produces an equality in $\zg$ of the following type $$ m (u+j) + i + bp \= 0$$
with $m \in \Mnz$, $u \in \Mu$, $j \in \Irn$, $i \in \Izero$ and $b \in \zg $.

c) In the theory of valued fields, a
dynamical proof of the fact $\Vr(p)$ from the presentation ${\cal K}$
produces an equality in $\zg$ of the following type $$ m ( (u+j)p^{n+1} + a_n
p^n +\ldots + a_1 p + a_0 ) + i \= 0$$ with $m \in \Mnz$, $u \in \Mu$,
$j \in \Irn$, the $a_k \in \Vvr$ and $i\in \Izero$.

d) In the theory of valued fields, a dynamical proof of the fact $\Rn(p)$ from
the presentation ${\cal K}$
produces an equality in $\zg$ of the following type $$ m ( (u+j)p^{n+1} + j_n
p^n +\ldots + j_1 p + j_0 ) + i \= 0$$ with $m \in \Mnz$,
$u \in\Mu$, $j$
and the $j_k$ in $\Irn$ and $i \in \Izero$.

e) In the theory of valued fields, a dynamical proof of the fact $\U(p)$ from
the presentation ${\cal K}$
produces an equality in $\zg$ of the following type $$ m ( (u+j) p^{n+1} + a_n
p^n +\ldots + a_1 p + (u' + j') ) + i \= 0 $$ with $m \in \Mnz$,
$u,u'\in \Mu$, $j,j' \in \Irn$, the $a_k$ in $\Vvr$ and $i \in \Izero$.
\end{proposition}

\begin{proof}{Proof:} We use the same letter notations as in Lemma \ref{Lemme
4.4}.

We get a) and b) as immediate consequences of Lemma
\ref{lem-valf-provfact-colaps}
a) and b).

In c), d) and e) the stated conditions are sufficient because valued fields
have
good simplification axioms
(cf.\ Lemmas \ref{lem-valf-valr} and
\ref{lem-valr}). Let us see that they are necessary conditions.

For c), use Lemma \ref{lem-valf-provfact-colaps} c) and write the collapse of
the presentation ${\cal K} \cup (zp-1=0,\Rn(z))$. We get an equality $\ (u_1+
j_1+ z a_1(z)) + i_1(z) \= (pz - 1)b_1(z)\ $. Let $n$ be the maximum
of the z-degrees of $z a_1(z)$ and $i_1(z)$. Multiply the equality by $p^n$
and
replace in the left-hand side
each $p^kz^k$ by $1$ modulo $(zp - 1)$.
After this transformation the right-hand side becomes $0$ and we get an
equality
$\ m_1 ( (u_1 + j_1) p^n + a_2(p) ) + i_2 \= 0\ $ where the $p$-degree of
$a_2$
is $\le n$. So we are done.

Same proof for d).

For e) we have two equalities from collapses: $\ m_1 ((u_1+j_1)p^{n+1}+
a_{1,n}p^n+ \cdots+ a_{1,1}p+ a_{1,0})+ i_1 \= 0\ $ and $\ m_2 ( a_{2,m} p^m
+\cdots + a_{2,1} p + (u_2+j_2)) + i_2 \= 0\ $. We assume w.l.o.g. that $m_1 =
m_2 = m$, Multiply the first equality by a convenient power of $p$ (e.g.
$p^{n+1}$) and add the two equalities. \end{proof}

\begin{remark} \label{rem-clotint}
{\rm From c) we get as corollary a constructive version of the classical
theorem
saying that the intersection of valuation rings containing the subring $A$ of
a field $K$ is the integral closure of $A$ in $K$. }
\end{remark}
\begin{corollary} \label{cor-clotint}
Let $A$ be subring of a field $K$. Consider the presentation ${\cal K}$
obtained
from the diagram of $K$ adding $\Vr(a)$ for each $a\in A$. Let $u\in K$. Then
the fact $\Vr(u)$ is provable from ${\cal K}$ in the theory of
valued fields if and only if $u$ is in the integral closure of $A$ in $K$.
\end{corollary}

\begin{theorem} \label{Theoreme 4.11}
The theories of quasi-valued rings, valued fields and algebraically closed
valued
fields prove the same facts.
\end{theorem}

\begin{proof}{Proof:} Proposition \ref{prop-factvalf} gives necessary and
sufficient conditions for provable facts in the theory of valued fields.
It is thus sufficient to see that the necessary conditions are also sufficient
in the theory of quasi-valued rings.

The case $p=0$ is taken care of by the axioms
$$\begin{array}{rl}
x \not= 0, xy = 0 &\,\vdash\, y = 0\\
x^2 = 0 &\,\vdash\, x = 0\\
\end{array}$$

The case $p \not=0$ is taken care of by
$ xy \not= 0 \,\vdash\, x\not= 0$

The case $\Vr(p)$
is taken care of by
$$\begin{array}{rl}
x \not= 0, xy = 0 &\,\vdash\, y = 0 \\
x^{n+1}- \sum_{k=0}^n a_k x^k=0 , \Vr(a_n), \cdots, \Vr(a_0) &\,\vdash\,
\Vr(x)
\\
\Vr(xy), \U(x) &\,\vdash\, \Vr(y) \\
\end{array}$$

The case $\Rn(p)$ is taken care of by
$$\begin{array}{rl}
x \not= 0, xy = 0 &\,\vdash\, y = 0 \\
x^{n+1}-\sum_{k=0}^n a_k x^k=0, \Vr(a_n), \cdots, \Vr(a_0) &\,\vdash\, \Vr(x)
\\
\Rn(xy), \U(x) &\,\vdash\, \Rn(y)\\
\Rn(x^2) &\,\vdash\, \Rn(x) \\
\end{array}$$

The case $\U(p)$ is taken care of by
$$\begin{array}{rl}
x \not= 0, xy = 0 &\,\vdash\, y = 0 \\
x^{n+1} - \sum_{k=0}^n a_k x^k=0, \Vr(a_n), \cdots, \Vr(a_1), \U(a_0)
&\,\vdash\, \U(x) \qquad \\
\U(xy), \U(x) &\,\vdash\, \U(y) \\
\end{array}$$
\end{proof}

Finally, we get from Theorem \ref{Theoreme 4.10} (with the same proof as in Proposition \ref{prop-factvalf}) the following generalized Positivstellensatz.

\begin{theorem}[generalized
Positivstellensatz for algebraically closed valued
fields]\label{genPosval}

Let $(K,A)$ be a valued field and let $U_A$ be the invertible elements of $A$,
$I_A$ the maximal ideal of $A$.
Suppose that $(K',A')$ is an algebraically closed valued field extension of
$K$
(so that $A=A' \cap K$). Denote by $U_{A'}$ the invertible elements of $A'$,
$I_{A'}$ the maximal ideal of $A'$.

Consider five finite families $(\Rzero,\Rnz,\Rvr,\Rrn,\Ru)$ in the polynomial
ring
$K[x_1,x_2,\cdots,x_m]=K[x]$.

Let $\Izero$ be the ideal of $K[x]$ generated by $\Rzero$, $\Mnz$ the monoid of
$K[x]$ generated by $\Rnz$, $\Vvr$ the subring of $K[x]$ generated by $\Rvr
\cup
\Rrn \cup \Ru \cup A$, $\Irn$ the ideal of $\Vvr$ generated by $\Rrn\cup I_A$,
$\Mu$ the monoid generated by $\Ru\cup U_A$.

Let ${\cal S}$ be the set of $x\in {K'}^m$ such that
$n(x) = 0$ for $n \in \Rzero$,
$t(x) \not= 0 $ for $t \in \Rnz$,
$c(x) \in A' $ for $c \in \Rvr$,
$v(x) \in U_{A'}$ for $v \in \Ru$ and $k(x) \in I_{A'}$ for $k \in \Rrn$.

a) The polynomial $p$ is everywhere zero on ${\cal S}$ if and only if there is
an
equality
$$p^n m (u+j) + i \= 0$$
with $m \in \Mnz$, $u \in \Mu$, $j \in \Irn$ and $i \in \Izero$.

b) The polynomial $p$ is everywhere nonzero on ${\cal S}$ if and only if
there
is
an equality:
$$m (u+j) + i + bp \= 0$$
with $m \in \Mnz$, $u \in \Mu$, $j \in \Irn$, $i \in \Izero$ and $b \in K[x] $.

c) $p({\cal S}) \subset A'$ if and only if there is an equality
$$m ( (u+j) p^{n+1} + a_n p^n + \cdots+ a_1 p + a ) + i \= 0 $$ with $m \in
\Mnz$, $u \in \Mu$, $j \in \Irn$, the $a_k \in \Vvr$ and $i\in \Izero$.

d) $p({\cal S}) \subset I_{A'}$ if and only if there is an equality
$$m ( (u+j) p^{n+1} + j_n p^n + \cdots+ j_1 p + j ) + i \= 0$$ with $m\in
\Mnz$,
$u \in \Mu$, $j$ and the $j_k$ in $\Irn$ and $i \in \Izero$.

e) $ p({\cal S}) \subset U_{A'}$ if and only if there is an equality: $$m (
(u+j) p^{n+1} + a_n p^n + \cdots+ a_1 p + (u' + j') ) + i \= 0 $$ with $m \in
\Mnz$,
$u,u' \in \Mu$, $j,j' \in \Irn$, the $a_k$ in $\Vvr$ and $i \in \Izero$.
\end{theorem}

\subsection{Related results of Prestel-Ripoli} \label{subsec PrRi}

Our Positivstellensatz for valued fields is closely related to some results of
Prestel and Ripoli (cf. \cite{PrRi}). The paper \cite{CP}
of Coquand and Persson also contains another kind of formal
``IntegralvalueStellensatz".

The direct part of Theorem 3.1 in
\cite{PrRi} has the following consequence.

\begin{theorem}[integral-valued rational functions on algebraically closed
valued fields]\label{th PrRI} Let $(K,A)$ be a valued field, $U_A$ the group
of invertible elements of $A$ and $I_A$ the maximal ideal of $A$. Let
$(K',A')$
be an algebraic closure of $(K,A)$ as valued field (so that $A=A' \cap K$).
Let us
denote $A[x_1,x_2,\ldots,x_m]=A[x]$, $K(x_1,x_2,\ldots,x_m)=K(x)$ and
$\Irn=I_A[x]$ the ideal of $A[x]$ generated by~$I_A$.

\ni Assume that $K$ is dense in $K'$, i.e., the residue field is algebraically
closed and the value group is divisible.
Consider a rational function $f=f_1/f_2\in K(x)$ with $f_1$ and $f_2$ in
$A[x]$ ($f_2\not= 0$).

\ni Then the following two assertions are equivalent:

a) Whenever $\xi\in A^m$ and $f_2(\xi)\not= 0$ then $f(\xi)\in A$ (in this case
we write $f(A)\subset A$).

b) There exists an algebraic identity in $A[x]$: $$ (1+j)f_1 \= a f_2$$
with $j
\in I_A[x]$ and $a \in A[x] $ (in this case
$f(x)=a(x)/(1+j(x))$ and we write $f\in A[x]/(1+I_A[x])$) \end{theorem}

A. Prestel also told us that he had obtained a Nullstellensatz (characterizing
polynomials $g$ s.t. $g(\xi)=0$ whenever $\xi \in A^m$ and
$h_1(\xi)=\cdots=h_m(\xi)=0$) when $K$ is algebraically closed, by using the
same
techniques as in \cite{PrRi}.

We remark that in \cite{PrRi} the result is an abstract one, with no
constructive
proof. Let us deduce this result (as an algorithmic one) from our
Positivstellensatz.

First remark that it suffices to consider the case that $K$ is algebraically
closed.

In a more general case, with $K$ not necessarily dense in $K'$ we get the
slightly
more general following result, which is a ``rational version" of Theorem
\ref{th PrRI}.

\begin{theorem}[integral-valued rational functions on valued fields]\label{th
rational PrRI} Let $(K,A)$ be a valued field, $U_A$ the group of invertible elements of
$A$ and $I_A$ the maximal ideal of $A$. Let $(K',A')$ be an algebraic
closed valued field extension of $(K,A)$ (so that $A=A' \cap K$). Let us denote
$A[x_1,x_2,\ldots,x_m]=A[x]$, $K(x_1,x_2,\ldots,x_m)=K(x)$ and
$\Irn=I_A[x]$ the ideal of $A[x]$ generated by $I_A$.

\ni Consider a rational function $f=f_1/f_2\in K(x)$ with $f_1$ and $f_2$ in
$A[x]$ ($f_2\not= 0$).

\ni Then the two following assertions are equivalent:

a) Whenever $\xi\in A'^m$ and $f_2(\xi)\not= 0$ then $f(\xi)\in A'$ (i.e.,
$f(A')\subset A'$).

b) There exists an algebraic identity in $A[x]\;$: $$ (1+j)f_1 \= a f_2$$ with
$j \in I_A[x]$ and $a \in A[x] $ (i.e., $f\in
A[x]/(1+I_A[x])$)
\end{theorem}

\begin{proof}{Proof:} Clearly b)
implies a). Assume now a). The fact that $f(A')\subset A'$ means the same
thing
as the incompatibility of the following
system of ``sign conditions":
$$\Vr(\xi_1),\ \ldots,\
\Vr(\xi_m),\ f_2(\xi)\not= 0,\ \Rn(\zeta),\ f_2(\xi)=\zeta
f_1(\xi)
$$
From Theorem \ref{Theoreme 4.10} this
implies an equality in $K[x,z]$
$$sf_2^k(x)(u+j(x)+za(x,z))\=(f_2(x)-zf_1(x))\,b(x,z) $$
 with $s\not= 0$ in $K$, $u\in U_A$,
$j\in I_A[x]$, $a(x,z)\in A[x,z]$ and $b\in K[x,z]$. Multiplying by
$(su)^{-1}$
we
get an algebraic identity:
$$f_2^k(x)[1+j_1(x)+za_1(x,z)]\=(f_2(x)-zf_1(x))\,b_1(x,z) $$
 Let
$a_1(x,z)=a_{1,0}+a_{1,1}z+\cdots+a_{1,p-1}z^{p-1}$. Applying the Rabinowitch's
trick, we multiply by $f_1^{p}$, replace in the left-hand side $z^hf_1^h$ by
$f_2^h$ modulo $(f_2(x)-zf_1(x)) $, and we get an algebraic identity:
$$f_2^k(x)[(1+j_1(x))f_1^p+
a_{1,0}(x)f_1^{p-1}f_2+a_{1,2}(x)f_1^{p-2}f_2^2+\cdots+a_{1,p-1}(x)f_2^{p}]\
=0
$$
 We deduce:
 $$(1+j_1(x))f_1^p+
a_{1,0}(x)f_1^{p-1}f_2+a_{1,2}(x)f_1^{p-2}f_2^2+\cdots+a_{1,p-1}(x)f_2^{p}\= 0
$$
 i.e., $f=f_1/f_2$ is in the
integral
closure of $A[x]_S$ where $S$ is the monoid $1+I_A[x]$. It is well known that
$A[x]$ is integrally closed. So $A[x]_S$ is also integrally closed. And we get
what we want.
\end{proof}

\section{A Positivstellensatz for ordered groups}

Our theory is based on the purely equational theory of abelian groups. The
group law
 is denoted additively. We have $0$ as only constant. The purely
equational
theory of abelian groups can be put in unary form. The free abelian group
generated by a set of generators $G$ will be denoted by $\abg$. The unary
predicate is $x=0$. Terms are replaced by elements of $\abg$. We call
$\cl_{g}$
the unary language of abelian groups. There are three direct algebraic axioms:
$$\begin{array}{rlccc}
&\,\vdash\, 0= 0&\qquad&\qquad&D(1)_g \\ x = 0, y=0 &\,\vdash\, x+y =
0&\qquad&\qquad&D(2)_g \\ x = 0 &\,\vdash\, -x= 0&\qquad&\qquad& D(3)_g \\
\end{array}$$

If $H$ is an abelian group and $x_1,\ldots,x_n$ are variables, the abelian
group
generated by $H$ and $x_1,\ldots,x_n$ (i.e., the group of affine forms with
variables
$x_1,\ldots,x_n$, constant part in $H$ and coefficients in $\Z$) will be
denoted
by
$H\{x_1,\ldots,x_n\}$.

In the sequel we say group instead of abelian group.

The central theory we consider is the theory of {\em (abelian) ordered groups}.

The language $\cl_{og}$ of ordered groups is the unary language of abelian
groups $\cl_g$ with two more unary predicates $x\ge 0$ and $x>0$.

Axioms of {\em proto-ordered groups} are the following.
$$\begin{array}{rlccc}
x = 0 ,\ y \ge 0 &\,\vdash\, x+y \ge 0 &\qquad&\qquad& D(1)_{og} \\ x \ge 0
,\ y
\ge 0 &\,\vdash\, x+y \ge 0&\qquad&\qquad& D(2)_{og} \\ &\,\vdash\, 0 \ge
0&\qquad&\qquad& D(3)_{og} \\
x=0,\ y > 0& \,\vdash\, x+y > 0 &\qquad&\qquad&D(4)_{og}\\ x > 0,\ y\ge 0&
\,\vdash\, x+ y > 0&\qquad&\qquad& D(5)_{og}\\
x > 0 &\,\vdash\, x \ge 0&\qquad&\qquad& D(6)_{og} \\ 0 > 0 &\,\vdash\, \bot
&\qquad&\qquad& C_{og} \\ \end{array}$$

Axioms of {\em ordered groups}
are axioms of proto-ordered groups and
the three
axioms
$$\begin{array}{rlccc}
x \ge 0,\ -x \ge 0 & \,\vdash\, x= 0 &\qquad&\qquad&S(1)_{og}\\ &\,\vdash\, x
\ge 0 \ \vee \ -x \ge 0 &\qquad&\qquad&Dy(1)_{og} \\ x\ge 0 &\,\vdash\, x =
0 \
\vee \ x >0&\qquad&\qquad& Dy(2)_{og} \\ \end{array}$$

\begin{remark}{\rm
Here also we can see a structure similar to that outlined in Remark
\ref{ordre-des-predicats}. The order on the predicates is $=0$,
$\geq0$,
$>0$. The axiom
$D(6)_{og}$ is an inclusion axiom, and the other direct algebraic axioms are
construction axioms.}
\end{remark}

A {\em divisible ordered group} is an ordered group satisfying the following
dynamical axioms (one for each integer $n > 1$): $$ \,\vdash\, \exists y \
n y =
x \ \ \ \ \ Dy_n(3)_{og}$$

We have easily the following results.

\begin{proposition} \label{prop-coll-preordgrp} Let ${\cal
H}=(G;\Rzero,\Rnng,\Rpos)$ be a presentation in the language $\cl_{og}$.
Let $\Hzero$ be the subgroup of $\abg$ generated by $\Rzero$, and $\Pnng$ the
additive monoid in $\abg$ generated by $\Rnng\ \cup \ \Rpos$.
A collapse of the presentation ${\cal H}$ in the theory of proto-ordered
groups
produces an equality in $G$ $$s + q + i \= 0$$ with $s\in \Rpos$, $q\in \Pnng$
and $i\in \Hzero$. \end{proposition}

\begin{lemma} \label{lem-coll-sim-ordgrp} Let ${\cal
H}=(G;\Rzero,\Rnng,\Rpos)$
be a presentation in the language $\cl_{og}$, $p \in \abg$ and $y$ a new
variable.

a) If the presentation
${\cal H}\cup ( p= 0 )$
collapses in the theory of proto-ordered groups, then so does the presentation
${\cal H}\cup ( p\ge 0, \ -p\ge 0 )$.

b) If the presentations
${\cal H}\cup ( p\ge 0 )$ and
${\cal H}\cup ( -p\ge 0 )$ collapse in the theory of proto-ordered groups,
then
so does the presentation ${\cal H}$.

c) If the presentations
${\cal H}\cup ( p>0 )$ and
${\cal H}\cup ( p=0 )$ collapse in the
theory of proto-ordered groups,
then so does the presentation ${\cal H}\cup ( p\ge 0 )$.

d) If the presentation ${\cal H}\cup ( ny-p=0 )$ collapses in the
theory of proto-ordered groups,
then so does the presentation ${\cal H}$. \end{lemma}

\begin{proposition} \label{prop-coll-sim-ordgrp} The theories of proto-ordered
groups, ordered groups and divisible ordered groups collapse simultaneously.
\end{proposition}

\begin{proposition}[non-constructive formal Positivstellensatz]
\label{Proposition
formalPstgrpord}
Let $H$ be an
abelian group,\
$\Rzero$, $\Rnng$ and $\Rpos$ three families of elements of $H$. Let
$\Hzero$ be
the subgroup of $H$ generated by $\Rzero$, and $\Pnng$ the additive monoid in
$H$ generated by
$\Rnng\ \cup \ \Rpos$. Then the following properties are equivalent:

i) There exist $s\in \Rpos$, $q\in \Pnng$ and $i\in \Hzero$ with $s + q+ i = 0$
in $H$

ii) There exists no homomorphism $\phi : H \rightarrow L$ with $L$ a divisible
ordered group, $\phi (a)=0$ for $a \in \Rzero$, $\phi (p) \geq 0$ for $p\in
\Rnng$ and $\phi (s) > 0$ for $s\in \Rpos$.

iii) There exists no ordering of any quotient group $H/H_0$ with $\Rzero\subset
H_0$, $\Rpos+ H_0\subset (H/H_0)^{>0}$ and $\Rnng+H_0\subset (H/H_0)^{\ge 0}$.

\end{proposition}

\subsubsection*{Decision algorithm and constructive Positiv\-stellensatz}
The next theorem is easily obtained by a close inspection of a decision
algorithm
for testing emptiness in the theory of divisible ordered groups.

\begin{theorem}\label{th-eqtf-divordgroup} Let $H$ be an ordered group, $D$
its
divisible ordered closure, and $\Rzero,\Rnng,\Rpos$ three finite families of
$H\{x_1,x_2,\cdots,x_m\}=H\{x\}$. The system of sign conditions
$[ u(x) >0, q(x) \geq 0, j(x) =0]$ for
$u \in \Rpos,\ q \in \Rnng,\ j \in \Rzero$ is impossible in $D^n$
if and only if the presentation
$\dg(H) \ \cup $
$(\{x_1,x_2,\cdots,x_m\};\Rzero,\Rnng,\Rpos)$ collapses in the theory of
divisible ordered groups.
\end{theorem}

\begin{proof}{Proof:} (sketch)

Call $x$ the variable you want to eliminate in an existential assertion for a
system of signs conditions. Call $y$ the other variables considered as
parameters.
Write every sign condition in form
$nx=t(y)$ with $n\in \Z ^{\ge 0}$ and $t(y)\in H\{y\}$, or $nx>t(y)$ with
$n\in
\Z$ and $t(y)\in H\{y\}$, or $nx\ge t(y)$ with $n\in \Z$ and $t(y)\in H\{y\}$.
Any sign condition is equivalent to the same one multiplied by a positive
integer.

If there is one sign condition of the first type and $n>0$, multiply all other
conditions by $n$ and then substitute $nx$ by the value given by the first
sign
condition. So you get an equivalent system with $x$ in only the first
equality.
The existence of $x$ in $D$ is equivalent to the other conditions without $x$.

In the other case, you may assume w.l.o.g. that you have, for all sign
conditions with $n\not= 0$ the same absolute value for $n$. For example you
have
the system
$(t_1\ge nx, t_2\ge nx, t_3> nx, nx\ge t_4, nx> t_5)$ (S) (and other
conditions
without $x$).
The existence of an $x$ in $D$ verifying (S) is equivalent to a big
disjunction
of systems ``without $x$", each disjunct saying in what order are
the $t_i$. E.g. one of these disjuncts is $(t_1\ge t_2, t_2\ge t_3, t_3> t_5,
t_5\ge t_4)$. Clearly there is a covering of $\dg(H) \ \cup \
(\{y,x\};\Rzero,\Rnng,\Rpos)$ in the theory of divisible ordered groups
corresponding to this equivalence.

So eliminating one variable after the other, you get a covering where every
leaf
contains only conditions in $H$. If the system is impossible you have a
collapse of the presentation in the theory of divisible ordered groups. In
the
other case you may construct a point in $D^m$ corresponding to a leaf of your
tree. \end{proof}

Then, gluing Theorem \ref{th-eqtf-divordgroup}, Proposition
\ref{prop-coll-sim-ordgrp} and Proposition \ref{prop-coll-preordgrp} we get the ``baby Positivstellensatz''.

\begin{theorem}\label{babypositiv}
Let $H$ be an ordered group, $D$ its divisible ordered closure, and
$\Rzero,\Rnng,\Rpos$
three finite families of $H\{x_1,x_2,\cdots,x_n\}=H\{x\}$. Let $\Hzero$ be the
subgroup of $H\{x\}$ generated by $\Rzero$, and $\Pnng$ the additive monoid in
$H\{x\}$ generated by
$\Rnng\ \cup \ \Rpos\ \cup \ H^{>0}$.

The system of sign conditions
$[ u(x) >0, q(x) \geq 0, j(x) =0]$ for
$u \in \Rpos,\ q \in \Rnng,\ j \in \Rzero$ is impossible in $D^n$
if and only if there is an equality in $H\{x\}$ $$s + q + i = 0$$ with $s\in
\Rpos\cup H^{>0}$, $q\in \Pnng$ and $i\in \Hzero$. \end{theorem}

We give now a variant.

\begin{theorem}\label{babypositiv-variant} Let $H$ be an ordered group,
$D$ its divisible ordered closure, and $\Rzero,\Rnng,\Rpos$ three finite
families of $H\{x_1,x_2,\cdots,x_n\}=H\{x\}$ and $p\in H\{x\}$.
Let $\Hzero$ be
the subgroup of $H\{x\}$ generated by $\Rzero$, and $\Pnng$ the
additive monoid in $H\{x\}$ generated by $\Rnng\ \cup \ \Rpos\ \cup \ H^{>0}$.

Let ${\cal S}\subset D^n$ the ``semialgebraic" set $\{x\in D^n; u(x) >0, q(x)
\geq 0, j(x) =0$ for $u \in \Rpos,\ q \in \Rnng,\ j \in \Rzero\}$.

a) $p$ is positive on ${\cal S}$
if and only if there is an equality in $H\{x\}$ $$s + q + i = m p$$ with $s\in
\Rpos\cup H^{>0}$, $q\in \Pnng$, $i\in \Hzero$ and $m$ a nonnegative
integer.

b) Assume ${\cal S}$ to be nonempty, then $p$ is nonnegative on ${\cal S}$
if and only if there is an equality in $H\{x\}$ $$q + i = m p$$ with $q\in
\Pnng$, $i\in \Hzero$ and $m$ a positive integer.

c) Assume ${\cal S}$ to be nonempty, then $p$ is null on ${\cal S}$
if and only if there are two equalities in $H\{x\}$ $$q + i = m p\qquad {\rm
and}\qquad -q' + i' = m p $$ with $q, q'\in \Pnng$, $i, i'\in \Hzero$ and $m$ a
positive
integer.

d) $p$ is nonzero on ${\cal S}$
if and only if there is an equality in $H\{x\}$ $$s + q + i = m p$$ with $s\in
\Rpos\cup H^{>0}$, $q\in \Pnng$, $i\in \Hzero$ and $m$ an integer.

\end{theorem}

Finally we give an algebraic theory, the theory of {\em quasi-ordered groups}
which proves the same facts
as theories of ordered groups and ordered divisible groups.

Axioms are those of proto-ordered groups and the following simplification
axioms.

$$\begin{array}{rlccc}
x \ge 0,\ -x \ge 0 & \,\vdash\, x= 0 &\qquad&\qquad&S(1)_{og}\\ nx \ge
0&\,\vdash\, x \ge 0 &\qquad&\qquad&S_n(2)_{og} \\ nx > 0&\,\vdash\, x > 0
&\qquad&\qquad&S_n(3)_{og} \\ \end{array}$$

\begin{remark}\label{programation.lineaire} {\rm Previous results are closely
related to well known theorems in linear programming over $\Q$.
E.g., applying Theorem \ref{babypositiv} with $H=\Q$ and $\Rzero=\emptyset$ we get the Motzkin's transposition theorem (see \cite{Sch} Corollary 7.1k p.\ 94).
Similarly, when $H=\Q$ and
$\Rzero=\Rpos=\emptyset$ we get a variant of Farka's lemma (see \cite{Sch}
Corollary 7.1e p.89).
}
\end{remark}


\begin{thebibliography}{50}




\bibitem{BE72} Bastiani A., Ehresmann C.: \emph{Categories of sketched structures}.
Cahiers de topologie et geometrie differentielle 13 (2), 1972.

\bibitem{BCR87} Bochnak J., Coste M., Roy M.-F.: \emph{G\'eometrie Alg\'ebrique
R\'eelle}. Springer-Verlag. Ergeb. M. n 11. 1987.


\bibitem{CL84} Coppey L., Lair C.: \emph{Le\c{c}ons de th\'eorie des esquisses (I)}.
Diagramme {\bf 12}. Paris, 1984

\bibitem{CL88} Coppey L., Lair C.: \emph{Le\c{c}ons de th\'eorie des esquisses (II)}.
Diagramme {\bf 19}. Paris, 1988

\bibitem{CP} Coquand T., Persson H.:
\emph{Valuations and Dedekind's Prague Theorem}. To appear in J. Pure Appl. Algebra.

\bibitem{DDD85} Della Dora J., Dicrescenzo C., Duval D.: \emph{About a new method
for
computing in algebraic number fields}. In Caviness B.F. (Ed.):
        EUROCAL '85. Lecture Notes in
Computer Science 204,  289--290. Springer 1985.

\bibitem{DD89} Dicrescenzo C., Duval D.: \emph{Algebraic extensions and algebraic
closure
in Scratchpad}. In Gianni P., (Ed.):
         Symbolic and Algebraic Computation. Lecture Notes in
Computer Science 358,  440--446. Springer 1989.

\bibitem{Duv89} Duval, D.: \emph{Simultaneous computations in fields of arbitrary
characteristic}. In Kaltofen E., Watt S.M. (Eds.): Computers and
Mathematics, 321--
326. Springer 1989.

\bibitem{DG93} Duval D., Gonzalez-Vega L.: \emph{Dynamic evaluation and real
closure}. Symbolic computation, new trends and developments (Lille, 1993).
Math. Comput. Simulation {\bf 42} (1996), 551--560.

\bibitem{DR90} Duval D., Reynaud J.-C.: \emph{Esquisses et Calcul.} Annales Universite
Caen. Vol VII (1989-1990) p 15-83. (J. Dehornoy, editeur)

\bibitem{DR931} Duval D., Reynaud J.-C.: \emph{Sketches and Computation (Part I)
Basic
Definitions and Static Evaluation}. Mathematical Structures in Computer Science
{\bf 4} (1994) 185--238.

\bibitem{DR932} Duval D., Reynaud J.-C.: \emph{Sketches and Computation (Part II)
Dynamic Evaluation and Applications}. Mathematical Structures in Computer
Science
{\bf 4} (1994) 239--271.

\bibitem{DS} Duval D., Senechaud P.: \emph{Sketches and Parametrization}. Theoretical
Computer Science, {\bf 123} (1) (1994) 117--130

\bibitem{Ehr68} Ehresmann C.: \emph{Esquisses et types de structures alg\'ebriques.}
Bulletin de l'Institut Polytechnique Iasi 14, 1968.

\bibitem{Gom93} Gomez-Diaz T.: \emph{Examples of using dynamic constructible
closure}.
Symbolic computation, new trends and developments (Lille, 1993).
Math. Comput. Simulation {\bf 42} (1996), 375--383.

\bibitem{Joyal} Joyal A.: \emph{Th\'eor\`eme de Chevalley-Tarski et remarque sur
l'alg\`ebre constructive}.
Cahiers de Topologie et G\'eom\'etrie Diff\'erentielle {\bf 16} (1975) 256--258

\bibitem{klp00} Kuhlmann F.-V., Lombardi H., Perdry H.:
\emph{Dynamic computations inside
the algebraic closure of a valued field}. Preprint 2000

\bibitem{Lom89} Lombardi H.: \emph{Th\'eor\`eme effectif des z\'eros r\'eel et
variantes}.
Publications Math\'ematiques Besan\c con. Th\'eorie des Nombres, 88--89.
Fascicule 1.
English abridged version: Effective real nullstellensatz and variants. p. 263-
288 in Effective Methods in Algebraic Geometry. Ed. Mora T., Traverso C.
Birkhauser 1991. Progress in Math. n 94.

\bibitem{Lom92} Lombardi H.: \emph{Une borne sur les degr\'es pour le Th\'eoreme des
z\'eros r\'eel effectif}. p 323--345. In: Real Algebraic Geometry. Proceedings,
Rennes 1991, Lecture Notes in Mathematics n 1524. Eds.: Coste M., Mahe L., Roy
M.-F. (Springer-Verlag, 1992).

\bibitem{Lom94} Lombardi H.: \emph{Relecture constructive de la th\'eorie d'Artin-
Schreier}. Annals of Pure and Applied Logic {\bf 91}, (1998), 59--92.

\bibitem{Lom96} Lombardi H.: \emph{Le contenu constructif d'un principe local- global
avec une application \`a la structure d'un module projectif de type fini.}
Publications Math\'ematiques de Besan\c con. Th\'eorie des nombres. Fascicule
94--95 \& 95--96, 1997.

\bibitem{Lom98b} Lombardi H.: \emph{Dimension de Krull, Nullstellens\"atze et
\'Evaluation dynamique}. Preprint 1998.

\bibitem{lom99} Lombardi H.:
\emph{Platitude, localisation et anneaux de Pr\"ufer~: une approche constructive.}
Preprint 1999.

\bibitem{lom99a} Lombardi H.:
\emph{Constructions cach\'ees en alg\`ebre abstraite (1) Relations de d\'ependance
int\'egrale}. Preprint 1999.

\bibitem{lc00} Lombardi H., Coquand T.:
\emph{Constructions cach\'ees en alg\`ebre abstraite (3) Dimension de Krull,
Going Up,
Going Down,\dots}. In preparation.

\bibitem{lq99} Lombardi H., Quitt\'e C.: \emph{Constructions cach\'ees en alg\`ebre
abstraite (2) Th\'eor\`eme de Horrocks, du local au global}. Preprint 1999.

\bibitem{LR90} Lombardi H., Roy M.-F.: \emph{Th\'eorie constructive \'el\'ementaire
des
corps ordonn\'es}. Publications Math\'ematiques de Besan\c con, Th\'eorie des
Nombres, 1990-1991.
English abridged version: Constructive elementary theory of ordered fields. p.
249-262. in Effective Methods in Algebraic Geometry. Ed. Mora T., Traverso C.
Birkhauser 1991. Progress in Math. n 94.

\bibitem{MR} Makkai M., Reyes G.: \emph{First order categorical logic}. Lecture Notes
in Mathematics 611, Springer 1977.

\bibitem{MRR88} Mines R., Richman F., Ruitenburg W.: \emph{A Course in Constructive
Algebra}. Springer-Verlag. Universitext. 1988.

\bibitem{PrRi} Prestel A., Ripoli C.: \emph{Integral valued rational functions on
valued fields}.
Manuscripta Math. {\bf 73}, 437--452 (1991)

\bibitem{Sch} Schrijver A.~: \emph{Theory of integer and linear programming}. John
Wiley. New-York. (1985)

\bibitem{Ste74} Stengle, G.: \emph{A Nullstellensatz and a Positivstellensatz in
semialgebraic geometry}. Math. Ann. 207, 87-97 (1974)

\bibitem{ST} Troelstra, A., Schwichtenberg, H.: \emph{Basic proof theory}.
Cambridge Tracts in Theoretical Computer Science, 43. Cambridge University
Press, Cambridge, 1996.

\bibitem{Wei84} Weispfenning, V.: \emph{Quantifier elimination and decision
procedures
for
valued fields}. In  M\"uller,~G.H., Richter,~M.M. (Eds.): Models and Sets.
Lecture Notes in Math.
 1103,
419--472. Springer 1984.

\end{thebibliography}
\end{document}